\documentclass[12pt,reqno]{amsart}
\usepackage{stmaryrd}
\usepackage[margin=1in]{geometry}
\usepackage{comment}
\usepackage[colorlinks,linkcolor=magenta,citecolor=cyan]{hyperref}
\usepackage{amssymb}
\usepackage{amsfonts}
\usepackage{amsmath}
\usepackage{amsthm}
\usepackage{color}
\usepackage{graphicx}
\usepackage{epsfig,mathrsfs}
\usepackage[bf,SL,BF]{subfigure}
\usepackage{fancyhdr}
\usepackage{graphicx}
\usepackage{subfigure}
\usepackage{caption}
\usepackage{wrapfig}
\usepackage{cases}
\usepackage{makecell}
\usepackage{multirow}
\usepackage{algorithm}
\usepackage{enumitem}
\usepackage{algorithmic}
\usepackage{comment}
\newcommand{\blue}{\textcolor{black}}
\newtheorem{theorem}{Theorem}
\newtheorem{lemma}[theorem]{Lemma}

\newtheorem{definition}[theorem]{Definition}
\newtheorem{example}[theorem]{Example}
\newtheorem{remark}{Remark}

\numberwithin{equation}{section}

\numberwithin{theorem}{section}

\numberwithin{subsection}{section}

\newcommand{\R}{\mathbb{R}}
\newcommand{\N}{\mathbb N}
\newcommand{\pa}{\partial}
\newcommand{\LO}[1]{L^{#1}(\Omega)}
\newcommand{\LQ}[1]{L^{#1}(Q_T)}
\newcommand{\LQtauT}[1]{L^{#1}(Q_{\tau,T})}

\newcommand{\bra}[1]{\left(#1\right)}
\newcommand{\sbra}[1]{\left[#1\right]}
\newcommand{\sumi}{\sum_{i=1}^{m}}
\newcommand{\intO}{\int_{\Omega}}
\newcommand{\eps}{\varepsilon}
\newcommand{\EE}{\mathscr{E}}

\newcommand{\vr}{\varrho}
\newcommand{\vpt}{\varphi_\tau}

\title[Reaction-diffusion systems with critical growth rates]{\blue{Analysis of mass controlled reaction-diffusion systems with nonlinearities having critical growth rates}}

\author{Chunyou Sun}
\address{Chunyou Sun \hfill\break
	School of Mathematics and Statistics, Lanzhou University Lanzhou, 730000, PR China}
\email{sunchy@lzu.edu.cn}
\author{Bao Quoc Tang}
\address{Bao Quoc Tang \hfill\break
	Institute of Mathematics and Scientific Computing, University of Graz,
	Heinrichstrasse 36, 8010 Graz, Austria}
\email{quoc.tang@uni-graz.at, baotangquoc@gmail.com}
\author{Juan Yang}
\address{Juan Yang \hfill\break
	Institute of Mathematics and Scientific Computing, University of Graz,
	Heinrichstrasse 36, 8010 Graz, Austria,
	\hfill\break
School of Mathematics and Statistics, Lanzhou University Lanzhou, 730000, PR China}
\email{yangjuan18@lzu.edu.cn, yangjuan0912@gmail.com}
\begin{document}

\begin{abstract}
	We analyze semilinear reaction-diffusion systems that are mass controlled, and have nonlinearities that satisfy critical
	growth rates. The systems under consideration are only assumed to satisfy natural assumptions, namely the preservation of non-negativity and a control of the total mass. It is proved in dimension one that if nonlinearities have (slightly super-) cubic growth rates then the system has a unique global classical solutions. Moreover, in the case of mass dissipation, the solution is bounded uniformly in time in sup-norm. One key idea in the proof is the H\"older continuity of gradient of solutions to parabolic equation with possibly discontinuous diffusion coefficients and low regular forcing terms.  When the system possesses additionally an entropy inequality, the global existence and boundedness of a unique classical solution is shown for nonlinearities satisfying a cubic intermediate sum condition, which is a significant generalization of cubic growth rates. The main idea in this case is to combine a modified Gagliardo-Nirenberg inequality and the newly developed $L^p$-energy method in \cite{morgan2021global,fitzgibbon2021reaction}. This idea also allows us to deal with the case of discontinuous diffusion coefficients in higher dimensions, which has \blue{only recently been touched} in the context of mass controlled reaction-diffusion systems.
\end{abstract}

\maketitle
\noindent{\small{\textbf{Classification AMS 2010:} 35A01, 35A09, 35K57, 35Q92.}}


\tableofcontents

\section{Introduction and main results}
Let $\Omega=(0,L)$ for $L>0$. We study \blue{the} reaction-diffusion system for vector of concentrations $u=(u_1,\cdots,u_m): \Omega\times \R_+ \to \R^m$,$m\geq 1$, \blue{given by}

\begin{equation}\label{eq1.1}
  \begin{cases}
    \partial_{t}u_i-d_i\partial_{xx}u_i=f_i(x,t,u), &x\in \Omega,~ t>0,\\
    \partial_xu_i(0,t)=\partial_xu_i(L,t)=0, &t>0,\\
    u_i(x,0)=u_{i,0}(x), &x\in\Omega,
  \end{cases}
\end{equation}
where $d_i>0$ are diffusion coefficients, the initial data are non-negative and bounded, i.e. $0\leq u_{i,0}\in L^{\infty}(\Omega), \forall i=1,\ldots, m$. \blue{The} nonlinearities satisfy the following assumptions:

\begin{enumerate}[label=(A\theenumi),ref=A\theenumi]

	\item\label{A1} (Local Lipschitz)  \blue{For all $i=1,\cdots, m$, $f_i:\Omega\times\mathbb{R}_+\times\mathbb{R}_+^m\rightarrow \mathbb{R}$ is locally Lipschitz continuous in the third argument, continuously differentiable in the second argument, and two times continuously differentiable in the first argument. }
		
	\noindent \blue{(Quasi-positivity) Moreover, they are quasi-positive, that is for any $i\in {1,\cdots ,m}$ and any $(x,t)\in\Omega\times\R_+$, it holds
	$$f_i(x,t,u)\geq 0 \text{ provided } u\in \mathbb{R}^m_+ \text{ and } u_i=0.$$}
	
	\item\label{A2} (Mass control)
	\begin{equation*}
	  \sum^m_{i=1}f_i(x,t,u)\leq k_0+k_1\sum^m_{j=1}u_j,~~\forall u\in \R_+^m, \; \forall (x,t)\in \Omega\times \R_+,
	\end{equation*}
	for some $k_0\geq 0, k_1\in\mathbb{R}.$
\end{enumerate}

The local Lipschitz continuity of nonlinearities in \eqref{A1} ensures the existence of a local classical solution in a maximal time interval. The quasi-positivity is a preservation of non-negativity\blue{. That is, if} the initial data are non-negative, then the solution is non-negative as long as it exists. This property has a simple physical interpretation\blue{. If a} concentration is zero at a time then it cannot be consumed in a reaction. Reaction-diffusion systems satisfying \eqref{A1} and \eqref{A2} appear naturally in modeling many real life phenomena, ranging from chemistry, biology, ecology, or social sciences. Remarkably, these two natural assumptions are not enough to ensure global existence of bounded solutions as it was pointed out by counterexamples \blue{in \cite{pierre2000blowup} and \cite{Pierre2023Examples}}. In fact, without further assumptions on the nonlinearities, it is not known if \eqref{eq1.1} possesses a \blue{global solution in any sense}. One main reason is that from \eqref{A1} and \eqref{A2} only limited a priori estimates have been derived. More precisely, by summing the equations of \eqref{eq1.1}, integrating in $\Omega$ and using \eqref{A2}, one (formally) obtains
\begin{equation*}
	\frac{d}{dt}\sumi \intO u_i(x,t)dx \leq k_0|\Omega| + k_1\sumi \intO u_i(x,t)dx.
\end{equation*}
The standard Gronwall inequality implies for any $T>0$,
\begin{equation*}
	\sumi\intO u_i(x,t)dx \leq e^{k_1T}\sumi \intO u_i(x,0)dx + \frac{k_0}{k_1}|\Omega|e^{k_1T} \quad \forall t\in (0,T).
\end{equation*}
Thanks to the non-negativity of solutions implied by \eqref{A1}, we get
\begin{equation}\label{LinftyL1}
	u_i\in L^{\infty}(0,T;\LO{1})
\end{equation}
$\text{ for all } i=1,\ldots, m, \; \text{ and all } \; T>0.$
With this estimate, it was shown early in \cite{morgan1989global} that \eqref{eq1.1} has a global classical solution if the nonlinearities are bounded by
\begin{equation}\label{bounded}
	|f_i(u)| \leq C\left(1+\sum_{j=1}^m u_j\right)^r \quad \forall i=1,\ldots, m, \; \forall u\in \R_+^m,
\end{equation}
where the growth rate $r$ is sub-critical\footnote{One observes that this critical growth $r_{\text{critical}}$ is the same as the Fujita exponent for seminar heat equation, see e.g. \cite{fujita1966blowing}}
\begin{equation}\label{critical}
	1\leq r < r_{\text{critical}}:= 1 + \frac{2}{n},
\end{equation}
where $n\in \mathbb N$ is the spatial dimension, i.e. $\Omega\subset\R^n$. It is observed that in one dimension, the global existence is obtained for the nonlinearities that have \textit{sub-cubic} growth rates. For two and higher dimensions, the results in \cite{morgan1989global} require the nonlinearities to be \textit{sub-quadratic}. By using a \textit{duality method}, it can be shown, \blue{see e.g.\cite{morgan1989global}, \cite{morgan1990boundedness},} \cite{pierre2000blowup}, that
\begin{equation}\label{L2}
	u_i \in L^{2}(0,T;L^2(\Omega))
\end{equation}
for all $i=1,\ldots, m$, and all $T>0$. When the nonlinearities have at most quadratic growth rates, this estimate \blue{implies} global \textit{weak solutions} in all dimensions \cite{desvillettes2007global}. An improved duality technique in \cite{canizo2014improved} showed that \eqref{L2} can be slightly improved, i.e. there exists $\eps>0$ depending only on the domain and diffusion coefficients such that
\begin{equation}\label{L2+eps}
	u_i\in L^{2+\eps}(0,T;L^{2+\eps}(\Omega))
\end{equation}
for all $i=1,\ldots, m$, and all $T>0$. This allows the authors in \cite{canizo2014improved} to obtain global classical solutions for systems of type \eqref{eq1.1} with quadratic nonlinearities in dimension two. \blue{In fact, global existence for quadratic nonlinearities with space dimension $n\leq 2$ was in fact first proved in \cite{Morgan2004Global} even with quasilinear diffusion}. Systems of form \eqref{eq1.1} with quadratic nonlinearities are in fact of high interest and importance due to \blue{their} relevance in bimolecular reactions and population dynamics such as Lotka-Volterra or SIR systems. The global existence of classical \blue{solutions} for such systems in three and higher dimensions \blue{was} open until recently when it was proved affirmatively in \blue{four works \cite{caputo2019solutions, souplet2018global, fellner2020global,fellner2021uniform}}.

\medskip
The global existence of bounded solution to \eqref{eq1.1} with \textit{super-quadratic} nonlinearities is widely open. Up to our knowledge, all existing results are \textit{conditional}, in the sense that additional assumptions are imposed: for instance, when the diffusion coefficients are close to each other \cite{canizo2014improved,morgan2020boundedness}, when diffusion coefficients are large enough \cite{cupps2021uniform}, or when the initial data are small enough \cite{cupps2021uniform}. This is the main motivation of the this paper where we investigate the global existence of classical solution in one dimensional for \eqref{eq1.1} with nonlinearities having {\it cubic or (slightly) higher growth rates}.

\medskip
Our first main result is the following theorem.

\begin{theorem}[Global classical solutions with (slightly super-)cubic nonlinearities]\label{th1.1}
 Let $\Omega = (0,L)$ for some $L>0$. Assume \eqref{A1} and \eqref{A2}. Then there exists $\varepsilon>0$ depending on $\Omega, m, d_i, k_0$ such that if the nonlinearities satisfy
  \begin{enumerate}[label={\normalfont (A\theenumi)},ref=A\theenumi]
  	\setcounter{enumi}{2}
  	\item\label{A3}
  	\begin{equation*}
  		|f_i(x,t,u)| \leq C\bra{1+|u|^{3+\eps}} \quad \forall i=1,\ldots, m,\; \forall u\in \R_+^m,\; \forall (x,t)\in \Omega\times \R_+
  	\end{equation*}
  \end{enumerate}
  for some constant $C>0$, the system \eqref{eq1.1} possesses a unique global classical solution. Moreover, if $k_1 = k_0 = 0$, then
	  \begin{equation*}
	  	\sup_{i}\limsup_{t\to\infty}\|u_i(t)\|_{\LO{\infty}} < +\infty.
	  \end{equation*}
\end{theorem}

\begin{remark}\hfill
	\begin{itemize}
  \item It is clear that the results of Theorem \ref{th1.1} still hold if \eqref{A2} is replaced by
  \begin{equation*}
  	\sum^m_{i=1}\alpha_if_i(x,t,u)\leq k_0 + k_1 \sumi u_i \qquad \text{for all}\quad u\in \R_+^m, (x,t)\in \Omega\times\R_+,
  \end{equation*}
  for some $(\alpha_i)_{i=1,\cdots,m}\in(0,\infty)^m.$
	\item   \blue{It is interesting that if a convective term is added to any of the equations, then the arguments fall apart in Theorem \ref{th1.1}. This is due to the fact that a new H\"older continuity results does not seem to hold.}
\end{itemize}
\end{remark}
To prove Theorem \ref{th1.1}, we first consider the case where \eqref{A2} is replaced by a (seemingly) stronger assumption
\begin{enumerate}[label=(A\theenumi'),ref=A\theenumi']
	\setcounter{enumi}{1}
	\item\label{A2'} there exists $g\in \LQ{\infty}$ \blue{with} $\pa_xg\in \LQ{\infty}$   such that
	\begin{equation*}
		\sumi f_i(x,t,u) = g(x,t), \quad \forall u\in \R_+^m, \; \forall(x,t)\in \Omega\times\R_+.
	\end{equation*}
\end{enumerate}

We will prove the following theorem.
\begin{theorem}\label{th1.2}
Let $\Omega = (0,L)$ for some $L>0$. Assume  \eqref{A1}, \eqref{A2'} and \eqref{A3}. Then for any non-negative, bounded initial data, \eqref{eq1.1} has a unique global classical solutions. Moreover, if $g \equiv 0$, the solution \blue{is uniformly sup-norm bounded}, i.e.
	\begin{equation}\label{ubit}
		\sup_{i}\limsup_{t\to\infty}\|u_i(t)\|_{\LO{\infty}}<+\infty.
	\end{equation}
\end{theorem}

At the first glance, \eqref{A2'} is stronger that \eqref{A2}, and consequently Theorem \ref{th1.2} is weaker than Theorem \ref{th1.1}. However, with a change of variable and introducing a new equation to the system (see subsection \ref{sub2.2} or \cite{fellner2020global}), it can be shown that Theorem \ref{th1.1} can be implied from Theorem \ref{th1.2}.

\medskip
Let us now describe the ideas to prove Theorem \ref{th1.2}. \blue{We follow the approach of Kanel \cite{kanel1990solvability}, which was used there for the case of mass conservation, i.e. when (A2) is fulfilled with an equality sign and $k_0 =k_1 = 0$.} By summing the equations of \eqref{eq1.1}, it follows from \eqref{A2'} that
\begin{equation*}
	\pa_t\bra{\sumi u_i} - \pa_{xx}\bra{\sumi d_iu_i} = g(x,t).
\end{equation*}
Integrating this relation with respect to time on $(0,t)$ gives
\begin{equation}\label{q1}
	\sumi u_i(x,t) - \pa_{xx}\bra{\int_0^t\sumi d_iu_i(x,s)ds} = \sumi u_{i,0}(x) + \int_0^tg(x,s)ds=:G(x,t).
\end{equation}
By defining $v(x,t) = \int_0^t\sumi d_iu_i(x,s)ds$, it follows that
\begin{equation}\label{q1_1}
	\sup_{i=1,\ldots, m}\sup_{\Omega\times(0,T)}|u_i| \leq m\sbra{\sup_{\Omega\times(0,T)}|\pa_{xx}v| + \|G\|_{\LQ{\infty}}}.
\end{equation}

The equation \eqref{q1} can \blue{also be}  written as
\begin{equation}\label{equation_v}
	b(x,t)\pa_tv(x,t) - \pa_{xx}v(x,t) = G(x,t),
\end{equation}
where $b(x,t)$ is bounded from above and below by positive constants. One cornerstone of \cite{fellner2020global} is that \eqref{equation_v} implies $v$ is H\"older continuous with an exponent $\gamma \in (0,1)$. In case when the nonlinearities are bounded by (slightly super-)quadratic polynomials, i.e. $|f_i(u)| \lesssim 1 + |u|^{2+\eps}$, this H\"older continuity of $v$ allows us to ultimately estimate
\begin{equation*}
	\sup_{\Omega\times(0,T)}|\pa_{xx}v| \lesssim 1 + \sbra{\sup_{i=1,\ldots, m}\sup_{\Omega\times(0,T)}|u_i|}^{\frac{3+\eps}{4}+ \frac{1-\gamma}{2(2-\gamma)}}.
\end{equation*}
Inserting this into \eqref{q1_1} gives, thanks to $\frac{3+\eps}{4}+ \frac{1-\gamma}{2(2-\gamma)}<1$ for small enough $\eps>0$,
\begin{equation}\label{q1_2}
	\blue{\sup_{i=1,\ldots, m}\sup_{\Omega\times(0,T)}|u_i| \lesssim C\bra{T,\gamma, \|G\|_{\LQ{\infty}}, \text{initial data}},}
\end{equation}
which implies the global existence \blue{and uniform sup-norm boundedness} of \eqref{eq1.1} in the case of (slightly super-)quadratic nonlinearities in all dimensions \blue{(see \cite{fellner2020global, fellner2021uniform})}. In \textit{dimension one}, similar to \cite[Theorems 2 \& 3]{kanel1990solvability} in the mass conservation case, we will prove that \textit{the spatial derivative $\pa_xv$ is also H\"older continuous} with an exponent $\alpha \in (0,1)$. This improvement is the key element to  deal with (slightly super-) \textit{cubic nonlinearities} as in \eqref{A3}. Indeed, by using the H\"older continuity of $\pa_xv$ and \eqref{A3}, we can estimate
\begin{equation*}
	\sup_{\Omega\times(0,T)}|\pa_{xx}v| \lesssim 1 + \sbra{1+\sup_{i=1,\ldots,m}\sup_{\Omega\times(0,T)}|u_i|}^{\bra{2+\frac{\eps}{2}}\frac{1-\alpha}{2-\alpha-\delta}}
\end{equation*}
for some $0<\delta<\alpha$. Now, by using $\bra{2+\frac{\eps}{2}}\frac{1-\alpha}{2-\alpha-\delta} < 1$ for small enough $\eps>0$, the estimate \eqref{q1_2} follows, hence the global existence of classical solutions. Finally, to show the uniform-in-time boundedness of solution \eqref{ubit}, we study \eqref{eq1.1} on time \blue{intervals} of fixed length with the help of a cut-off function and show that the solution is bounded \blue{independent} of the intervals.

\medskip
The assumption \eqref{A3} requires that {\it all } nonlinearities are bounded by (slightly super-)cubic polynomials. Our second main result shows global existence of bounded solutions to \eqref{eq1.1} with nonlinearities having critical growth rates $r_{\text{critical}}$ and \blue{satisfying} an \textit{entropy inequality} and a \textit{cubic intermediate sum condition}. This condition means that \textit{only one nonlinearity} is assumed to be \blue{bounded by a} cubic polynomial, while the others just need to satisfy an intermediate sum condition of order three. Intermediate sum \blue{conditions} of various orders \blue{have} been studied in \cite{morgan1989global,morgan1990boundedness} and revisited in \cite{morgan2020boundedness,fitzgibbon2021reaction}. In general, an \textit{intermediate sum condition of order $r$} means that
\begin{enumerate}[label=(A\theenumi), ref=A\theenumi]
	\setcounter{enumi}{3}
	\item\label{A4} there exists a lower triangular matrix $A = (a_{ij}) \in \R^{m\times m}$ with non-negative elements and positive diagonal elements such that for any $i=1,\ldots, m$,
	\begin{equation*}
		\sum_{j=1}^{i}a_{ij}f_j(x,t,u) \leq C\bra{1+\sum_{j=1}^mu_j}^r \quad \text{ for all } \;u\in \R_+^m, \; (x,t)\in\Omega\times \R_+
	\end{equation*}
	where $C>0$ is a fixed constant.
\end{enumerate}
It is remarked that \eqref{A4} is significantly more general than \eqref{bounded} (see Example \ref{example}). Using the estimate \eqref{LinftyL1}, it was shown in \cite{morgan1989global,fitzgibbon2021reaction} that under the additional assumption \eqref{A4}, system \eqref{eq1.1} has a unique global classical solution if the growth rate $r$ is sub-critical as in \eqref{critical}. \blue{Note that these results do not allow to have quadratic growth in  dimensions higher than three.} The case of quadratic nonlinearities in two dimension has been recently \blue{revisited} in \cite{morgan2020boundedness} by utilizing the improved duality method first proved in \cite{canizo2014improved}. The case of cubic intermediate \blue{sums} has not been treated in the literature. In this work, we show global classical solutions to \eqref{eq1.1} in one dimension with nonlinearities satisfying a cubic intermediate sum condition and the so-called \textit{entropy inequality}:
\begin{enumerate}[label=(E),ref=E]
	\item \label{E} there exists $\mu_1, \ldots, \mu_m \in \R$ and $k_2, k_3\geq 0$ such that
	\begin{equation*}
		\sum_{i=1}^mf_i(x,t,u)(\log u_i + \mu_i)\leq k_2\sumi u_i(\log u_i + \mu_i-1) + k_3,
	\end{equation*}
	for all $u\in (0,\infty)^m$ and all $(x,t)\in\Omega\times\R_+.$
\end{enumerate}
This entropy inequality appears frequently in (bio-)chemical reactions and therefore has been studied extensively in the literature. By assuming \eqref{E}, \blue{it was proved in \cite{goudon2010regularity}} that \eqref{eq1.1} is globally well-posed in one and two dimensions, respectively, with cubic nonlinearities, i.e. $r=3$ in \eqref{bounded}, and quadratic nonlinearities, i.e. $r=2$ in \eqref{bounded}. When the nonlinearities are \textit{strictly sub-quadratic}, i.e. $r<2$ in \eqref{bounded}, the global existence of classical solution was shown in \textit{all dimensions}, \blue{see \cite{morgan1989global} for a separable Lyupanov approach and \cite{caputo2009global} for a De Giorgi approach}. A breakthrough was shown in \cite{fischer2015global} where the author showed global \textit{renormalized solutions} to \eqref{eq1.1} under \eqref{E} \textit{without any growth assumptions} on the nonlinearities. Our result in the following theorem shows global classical \blue{solutions} to \eqref{eq1.1} under a critical intermediate sum condition, i.e. $r=r_{\text{critical}} = 3$ in \eqref{A4} in one dimension, and the entropy inequality \eqref{E}, \blue{which} significantly generalizes previous results.

\begin{theorem}[Global classical solutions with cubic intermediate sum]\label{th1.3}
	Let $\Omega = (0,L)$ for some $L>0$. Assume \eqref{A1}, \eqref{E} and \eqref{A4} with $r=3$. Moreover, assume that there exist $\ell>0$ and $C>0$ satisfying for all $i=1,\ldots, m$,
	\begin{equation}\label{polynomial}
		f_i(x,t,u) \leq C\bra{1+\sum_{j=1}^{m}u_j^\ell}
	\end{equation}
	for all $u\in \R_+^m$ and all $(x,t)\in\Omega\times\R_+$.	Then for any non-negative, bounded initial data, \eqref{eq1.1} has a unique global classical \blue{solution}. Moreover, if $k_2 = k_3 = 0$ \blue{or $k_2<0$} in \eqref{E}, the solution \blue{is} bounded uniformly in time in sup-norm, i.e.
	\begin{equation*}
		\sup_{i}\limsup_{t\to\infty}\|u_i(t)\|_{\LO{\infty}}<+\infty.
	\end{equation*}
\end{theorem}
We emphasize that by assuming \eqref{E}, the mass control assumption \eqref{A2} can be relaxed. Moreover, the assumption \eqref{polynomial} indicates that the nonlinearities are bounded above by polynomials, but \textit{we do not impose any restriction on the growth rate $\ell$}.

\medskip
Our key idea in proving Theorem \ref{th1.3} is to combine a new $L^p$-energy method, see \cite{morgan2021global,fitzgibbon2021reaction} and a modified Gagliardo-Nirenberg inequality. The $L^p$-energy method deduces that, under the intermediate sum condition \eqref{A4}, one can choose for any $2\leq p \in \mathbb N$ coefficients $\theta_\beta$ such that (see Lemma \ref{lem1}) the energy function
\begin{equation*}
	\EE_p[u]:= \sum_{|\beta| = \beta_1+\ldots + \beta_m = p}\int_{\Omega}\bra{\theta_\beta\prod_{i=1}^{m}u_i^{\beta_i}}dx
\end{equation*}
satisfies
\begin{equation*}
	\frac{d}{dt}\EE_p[u] + \alpha_p\sumi \intO \left|\pa_x(u_i^{p/2}) \right|^2dx \leq C\bra{1+\sumi \intO u_i^{p-1+r}dx}.
\end{equation*}
For $r = 3$, choosing $p = 2$ yields
\begin{equation*}
	\frac{d}{dt}\EE_2[u] + \alpha_2\sumi \|u_i\|_{H^1(\Omega)}^2 \leq C\bra{1+\sumi \|u_i\|_{\LO{4}}^4}.
\end{equation*}
To deal with the terms involving \blue{the} $\LO{4}$-norm on the right-hand side, we first utilize the entropy condition \eqref{E} to obtain a bound of $\|u_i\log|u_i|\|_{\LO{1}}$, then \blue{apply a  modified} Gagliardo-Nirenberg inequality in one dimension to show that the $\LO{4}$-norm can be controlled by the $H^1(\Omega)$-norm and $\|u_i\log|u_i|\|_{\LO{1}}$. This leads to bounds of $\EE_2[u]$ and consequently $L^\infty(0,T;L^2(\Omega))$ bounds, which in \blue{turn gives} $L^\infty(0,T;L^\infty(\Omega))$ bounds  by considering general $2\leq p \in \mathbb N$.

\medskip
There is another advantage of the $L^p$-energy method: it allows us to deal with the case of discontinuous diffusion coefficients. More precisely, consider a variant of \eqref{eq1.1} in arbitrary dimension, i.e. $\Omega \subset \R^n$, $n\geq 1$, with possibly discontinuous diffusion coefficients
\blue{\begin{equation}\label{discontinuous_diffusion}
\begin{cases}
	\pa_t u_i - \nabla_x\cdot (D_i(x,t)\nabla_x u_i) = f_i(x,t,u), & x\in\Omega, \; t>0,\\
	 u_i(x,t) = 0, & x\in\pa\Omega, t>0,\\
	u_{i,0}(x) = u_{i}(x,0), & x\in\Omega,
\end{cases}
\end{equation}}
where the diffusion matrix $D_i:\Omega\times[0,\infty)\rightarrow\mathbb{R}^{n\times n}$ satisfies
    \begin{equation}\label{eq-di}
		\lambda|\xi|^2\leq \xi^\top D_i(x,t)\xi,~\forall(x,t)\in\Omega\times[0,\infty),\forall\xi\in\mathbb{R}^n,\forall i=1,\cdots,m,
	\end{equation}
for some $\lambda>0$ and
    \begin{equation}\label{eq-d}
		D_i\in L^\infty_{\text{loc}}(\R_+;L^{\infty}(\Omega)),\forall i=1,\cdots,m.
	\end{equation}

Due to the low regularity of diffusion coefficients, one cannot expect classical \blue{solutions} to \eqref{discontinuous_diffusion}. The suitable framework is weak solutions. By using similar methods to the proof of Theorem \ref{th1.3}, we show that under \eqref{A1}, \eqref{E} and \eqref{A4}, the system \eqref{discontinuous_diffusion} has a unique global bounded weak solution. While all previous results focused on dimension one, the following theorem is proved in all dimensions.
\begin{theorem}\label{th1.4}
	Let $\Omega \subset \R^n$, $n\geq 1$. Assume \eqref{A1}, \eqref{E}, \eqref{eq-di},\eqref{eq-d} and \eqref{A4} with $r$ in \eqref{A4} fulfills
	\begin{equation*}
		1 \leq r \leq r_{\text{critical}} = 1 + \frac{2}{n}.
	\end{equation*}
   Moreover, assume that there exist $\ell>0$ and $C>0$ satisfying for all $i=1,\ldots, m$,
	\begin{equation}\label{eq-polynomial}
		f_i(x,t,u) \leq C\bra{1+\sum_{j=1}^{m}u_j^\ell}
	\end{equation}
	for all $u\in \R_+^m$ and all $(x,t)\in\Omega\times\R_+$.
	Then, for any non-negative, bounded initial data $u_0 \in L^{\infty}(\Omega)^{m}$, there exists a unique global bounded, non-negative solution to \eqref{discontinuous_diffusion}. Moreover, if $k_2 = k_3 = 0$ in \eqref{E}, the solution is bounded uniformly in time, i.e.
	\begin{equation*}
		\sup_{i}\limsup_{t\to\infty}\|u_i(t)\|_{\LO{\infty}} < +\infty.
	\end{equation*}
\end{theorem}
\begin{remark}\hfill
	\begin{itemize}
		\item \blue{Obviously, Theorem \ref{th1.3} is a special case of Theorem \ref{th1.4} in dimension $n = 1$ when all diffusion coefficients are constants.}
		\item Theorem \ref{th1.4} improves results of \cite{fitzgibbon2021reaction} by allowing the critical value $r = 1 + \frac 2n$.
		\item In case of constant diffusion, Theorem \ref{th1.4} recovers results in \cite{morgan2020boundedness} in two dimensions concerning quadratic intermediate sum condition \blue{at the cost of \eqref{E}}. In higher dimensions, still with constant diffusion coefficients, results in \cite{morgan2020boundedness} allow intermediate sum \blue{conditions} of order \blue{$r<r_* \leq 1 + \frac{4}{n+2}$}, which is obviously better than $1+\frac 2n$ as soon as \blue{$n\leq 3$}. \blue{However one} novelty of Theorem \ref{th1.4} is that it deals with discontinuous diffusion coefficients, which are out of reach for the duality method used in \cite{morgan2020boundedness}.
	\end{itemize}
\end{remark}

\begin{example}\label{example}
	Consider the reversible reaction for three chemical species $\mathcal{U}, \mathcal{V}, \mathcal{W}$ as \blue{given by}
	\begin{equation*}
		\alpha \mathcal U + \beta \mathcal V \leftrightharpoons \gamma W
	\end{equation*}
	with stoichiometric coefficients $\alpha, \beta, \gamma \in \mathbb N$, and, for the sake of simplicity, \blue{assume the reaction rate constants are one}. By applying the mass action law, one obtains the following one dimensional reaction-diffusion system with $\Omega = (0,L)$
	\begin{equation}\label{reversible}
	\begin{cases}
		\pa_t u - \pa_{x}(d_u(x,t)\pa_xu) = f_1(u,v,w):= -\alpha\bra{u^\alpha v^\beta - w^{\gamma}}, &x\in\Omega,\\
		\pa_t v- \pa_{x}(d_v(x,t)\pa_xv) = f_2(u,v,w):= -\beta\bra{u^\alpha v^\beta - w^\gamma}, &x\in\Omega,\\
		\pa_t w - \pa_{x}(d_w(x,t)\pa_xw) = f_3(u,v,w):= \gamma\bra{u^\alpha v^\beta - w^\gamma}, &x\in\Omega,
	\end{cases}
	\end{equation}
	subject to homogeneous Neumann boundary \blue{conditions} and non-negative, bounded initial data. It is easy to see that if $\gamma = 3$ and $\alpha, \beta \in \mathbb N$ arbitrary, the cubic intermediate sum condition is satisfied for \eqref{reversible}. Moreover, thanks to the reversibility, the entropy inequality condition \eqref{E} is also fulfilled since
	\begin{equation*}
		f_1\times \log u + f_2 \times \log v + f_3\times \log w = -\bra{u^\alpha v^\beta - w^\gamma}\bra{\log\bra{u^\alpha v^\beta} - \log \bra{w^\gamma}} \leq 0.
	\end{equation*}
	Therefore, one can apply Theorem \ref{th1.4} to obtain global existence and uniform-in-time \blue{bounds} of a unique weak solution to \eqref{reversible}. When the diffusion coefficients $d_u, d_v, d_w$ are smooth or constants, the solution is classical. We emphasize that global classical solutions to \blue{systems} of type \eqref{reversible} have been studied many times before, see e.g. \cite{laamri2011global,fellner2016exponential,pierre2017asymptotic,morgan2020boundedness}, but none of this results \blue{are applicable to the case $\gamma=3$ and $\alpha, \beta \in \mathbb N$ arbitrary, even with constant diffusion coefficients, unless $\alpha=\beta=1$ and $\gamma$ is arbitrary.}
\end{example}

\noindent\textbf{Organization of the paper.} In the next section, we show the global existence and boundedness of solutions for systems with cubic nonlinearities in one dimension. We prove Theorem \ref{th1.2} in section \ref{sub2.1}. The proof of Theorem \ref{th1.1}, as a consequence of Theorem \ref{th1.2}, is presented in section \ref{sub2.2}. Section \ref{sec3} is devoted to case of nonlinearities satisfying \blue{a cubic intermediate sum condition}, where systems with constant diffusion and discontinuous coefficients are considered in subsection \ref{sub3.1} and \ref{sub3.2} respectively. 

\noindent\textbf{Notation.} In this paper we will use the following notation, some of which will be \blue{recalled from time to time}:
\begin{itemize}
	\item For $T>0$ and $p\in [1,\infty]$, $Q_T:= \Omega\times(0,T)$ and
	\begin{equation*}
		L^p(Q_T):= L^p(0,T;L^p(\Omega))
	\end{equation*}
	equipped with the usual norm
	\begin{equation*}
		\|f\|_{\LQ{p}}:= \bra{\int_0^T\intO |f|^pdxdt}^{1/p}
	\end{equation*}
	for $1\leq p < \infty$ and
	\begin{equation*}
		\|f\|_{\LQ{\infty}}:= \underset{(x,t)\in Q_T}{\text{\normalfont ess sup}}|f(x,t)|.
	\end{equation*}
	\item For  $p\in [1,\infty]$, $\tau \geq 0$ and $\delta>0$, we denote by
	\blue{\begin{equation*}
		Q_{\tau,\tau+\delta}:= \Omega\times(\tau,\tau+\delta),
	\end{equation*}}
	and
	\begin{equation*}
		L^p(Q_{\tau,\tau+\delta}):= L^p(\tau,\tau+\delta; L^p(\Omega)).
	\end{equation*}
\end{itemize}

\section{Proof of Theorem \ref{th1.2}}\label{sub2.1}

\subsection{Preliminaries}
We start with the definition of classical solutions.

Considering \eqref{eq1.1} in arbitrary dimension, i.e. $\Omega \subset \R^n$, $n\geq 1$, be a bounded domain. The reaction-diffusion system for vector of concentrations $u=(u_1,\cdots,u_m): \Omega\times (0,T) \to \R^m$, $m\geq 1$, given by
\begin{equation}\label{eq}
  \begin{cases}
    \blue{\partial_{t}u_i-d_i\Delta u_i=f_i(x,t,u)}, &(x,t)\in \Omega\times(0,T),\\
    \nabla_xu_i\cdot \nu=0, &(x,t)\in \partial\Omega\times(0,T),\\
    u_i(x,0)=u_{i,0}(x), &x\in\Omega,
  \end{cases}
\end{equation}
where  $d_i>0$ are diffusion coefficients.
\begin{definition}[Classical solutions]
	Let $0<T\leq\infty$. A classical solution to \eqref{eq} on (0,T) is a vector of concentrations $u=(u_1,\cdots,u_m),m\geq1$, satisfying for all $i=1,\cdots,m$, $u_i\in C([0,T];L^p(\Omega))\cap L^\infty((0,T)\times\Omega)\cap C^{1,2}((\tau,T)\times\bar{\Omega})$ for all $p>1$ and all $0<\tau<T$, and $u$ satisfies each equation in \eqref{eq} pointwise in $Q_T$.
\end{definition}
\begin{theorem}[Local existence, \cite{fellner2020global}, Proposition 3.1]\label{local}
	Assume \eqref{A1}. Then, for any bounded, nonnegative initial data, \eqref{eq} possesses a local nonnegative classical solution on a maximal interval $[0,T_{\max})$. Moreover, if
	\begin{equation}\label{eq-blow-up}
	\begin{split}
	\limsup_{t\rightarrow T^-_{\max}}\|u_i(t)\|_{L^\infty(\Omega)}<\infty \quad \text{for all}~i=1,2,\cdots,m,
	\end{split}
	\end{equation}
	then $T_{\max}=+\infty.$
\end{theorem}

Thanks to Theorem \ref{local}, the global existence of strong solutions to \eqref{eq1.1} follows if we can show that
\begin{equation}\label{criterion}
\sup_{i=1,\ldots, m}\sup_{T\in(0,T_{\max})}\|u_i\|_{\LQ{\infty}} < +\infty.
\end{equation}
Moreover, due to the smoothing effect, we can shift the initial time to $0<\tau<T_{\max}$ to assume w.l.o.g. that {\textit{the initial data $u_{i,0}\in C^2(\bar{\Omega})$, $i=1,\ldots, m$,}} and satisfy the compatibility condition $\nabla u_{i,0}\cdot \nu = 0$ on $\pa\Omega$. \textbf{We will use these regular initial data for the rest of this paper.}

\medskip
The following interpolation lemma was proved in \cite{fellner2020global} and it holds in {\it all dimensions.}
\begin{lemma}[Regularity Interpolation, Neumann boundary conditions]\cite{kanel1990solvability, fellner2020global}\label{lem2.3}
Let $\Omega\subset\R^n$, $n\geq 1$, be a bounded domain with smooth boundary $\pa\Omega$. For some constant $d>0$, let $u$ be the solution to the inhomogeneous linear heat equation
  \begin{equation*}
    \begin{cases}
      \partial_{t}u-d\Delta u=\phi(x,t), &(x,t)\in \Omega\times(0,T),\\
      \nabla u\cdot\nu=0,  &(x,t)\in \partial\Omega\times(0,T),\\
      u(x,0)=u_0(x), &x\in\Omega,
    \end{cases}
  \end{equation*}
  Assume that there exists $\gamma\in[0, 1)$ such that for all $x, x'\in\Omega$, and all $t\in(0, T)$,
  $$|u(x,t)-u(x',t)|\leq H|x-x'|^\gamma.$$
  Then, the following gradient estimate follows:
  $$\sup_{Q_T}|\nabla u(x,t)|\leq C\|u_0\|_{C^1(\Omega)}+BH^{\frac{1}{2-\gamma}}\|\phi\|_{\LQ{\infty}}^{\frac{1-\gamma}{2-\gamma}},$$
  where $B>0$ and $C>0$ are constants depending only on $\Omega, d$ and $\gamma$.

\end{lemma}

We have a similar lemma for the case of homogeneous Dirichlet boundary conditions.
\begin{lemma}[Regularity Interpolation, Dirichlet boundary conditions]\label{lem-Diri} Let $\Omega\subset\R^n$, $n\geq 1$, be a bounded domain with smooth boundary $\pa\Omega$. Let $0\leq \tau < T$. For some constant $d>0$, let $u$ be the solution to the inhomogeneous linear heat equation
  \begin{equation*}
    \begin{cases}
      \partial_{t}u-d\Delta u=\phi(x,t), &(x,t)\in Q_{\tau,T},\\
      u(x,t)=0,  &(x,t)\in \partial\Omega\times(\tau,T),\\
      u(x,\tau)=u_0(x),  & x\in\Omega.
    \end{cases}
  \end{equation*}
 \blue{ Assume there exists} $\gamma\in[0, 1)$ such that for all $x, x'\in\Omega$, and all $t\in(\tau, T)$,
  $$|u(x,t)-u(x',t)|\leq H|x-x'|^\gamma.$$
  Then, the following uniform gradient estimate follows:
  \begin{equation}
  \label{2star}
   \sup_{Q_{\tau,T}}|\nabla u(x,t)|\leq C\|u_0\|_{C^1(\Omega)}+C_{T-\tau}H^{\frac{1-\delta}{2-\gamma-\delta}}\|\phi\|_{L^{\infty}(Q_{\tau,T})}^{\frac{1-\gamma}{2-\gamma-\delta}},
    \end{equation}
  where \blue{$C$ and $C_{T-\tau}$} are constants depending only on $T-\tau>0$, \blue{$\Omega, d$, $\gamma$ and $\delta\in(0,2-\gamma)$}.
\end{lemma}
\begin{remark}
	Using similar method as in \cite{fellner2020global} it is possible to show \eqref{2star} for $\delta = 0$ and the constant $C_T$ on the right hand is independent of $T$. The only modification is that one needs gradient estimates of the \blue{Green function} with homogeneous Dirichlet, instead of Neumann, boundary conditions. Here we present a simpler proof for \eqref{2star} using maximal-regularity \blue{of the parabolic equation}. Note that \eqref{2star} is enough for our later purpose by choosing $\delta \leq \gamma/3$.
\end{remark}

\begin{proof}
   Let $u=v+w$, where
   \begin{equation}\label{eq_v}
	\begin{cases}
		\partial_{t}v-d\Delta v=0, &(x,t)\in Q_{\tau,T},\\
      v(x,t)=0,  &(x,t)\in \partial\Omega\times(\tau,T),\\
      v(x,\tau)=u_0(x),  & x\in\Omega,
	\end{cases}
	\end{equation}
   and
   \begin{equation}\label{eq_w}
	\begin{cases}
	  \partial_{t}w-d\Delta w=\phi(x,t), &(x,t)\in Q_{\tau,T},\\
      w(x,t)=0,  &(x,t)\in \partial\Omega\times(\tau,T),\\
      w(x,\tau)=0,  & x\in\Omega.
	\end{cases}
	\end{equation}
   From equation \eqref{eq_v}, we can obtain
    \begin{equation}\label{eq_C^1}
	\|v(t)\|_{C^1(\Omega)}\leq C\|u_0\|_{C^1(\Omega)}.
	\end{equation}
   From equation \eqref{eq_w}, using $L^p$-max-regularity \cite{lamberton1987equations}, we have
   \begin{equation}
	\|w\|_{W^{1,2}_p(Q_{\tau,T})}\leq C_p\|\phi\|_{L^p(Q_{\tau,T})},\quad 1<p<\infty
	\end{equation}
	where the space $W^{1,2}_p(Q_{\tau,T})$ is defined as
	\begin{equation*}
		W^{1,2}_p(Q_{\tau,T}):= \left\{f\in L^{p}(Q_{\tau,T}):\, \|f\|_{W^{1,2}_p(Q_{\tau,T})}:= \sum_{2r+s\leq 2}\|\pa_t^r\pa_x^sf\|_{L^{p}(Q_{\tau,T})} < +\infty\right\}.
	\end{equation*}
   We have the embedding (see \cite[Lemma II.3.4]{Ladyzenskaja1968Linear}), for all $t\in (\tau,T)$,
      \begin{equation}
	\|w(t)\|_{W^{2(1-\frac 1p),p}(\Omega)}\leq C_{p,T}\|w\|_{W^{1,2}_p(Q_{\tau,T})},\quad 1<p<\infty.
	\end{equation}
   Moreover, by choosing $p> \frac{n+2}{2-\gamma} > \frac{n+2}{2}$ we have
    \begin{equation}
	W^{2(1-\frac 1p),p}(\Omega)\hookrightarrow W^{2-\delta,\infty}(\Omega),
	\end{equation}
   where
    \blue{\begin{equation}
	\frac{1}{\infty}<\frac{1}{p}-\frac{2(1-\frac{1}{p})-2+\delta}{n} \quad \Leftrightarrow \quad \delta < \frac{n+2}{p} < 2 - \gamma.
	\end{equation}}
    Using interpolation, we have
     \begin{equation}
     \begin{split}
      \|w\|_{C^1(\Omega)}&\leq C\|w\|_{C^{\gamma}(\Omega)}^\theta\cdot\|w\|_{W^{2-\delta,\infty}(\Omega)}^{1-\theta}\\
      &\leq C_p\|w\|_{C^{\gamma}(\Omega)}^\theta\cdot\|\phi\|_{L^p(Q_{\tau,T})}^{1-\theta}\\
      &\leq C_{p,T-\tau}\|w\|_{C^{\gamma}(\Omega)}^\theta\cdot\|\phi\|_{L^\infty(Q_{\tau,T})}^{1-\theta},
     \end{split}
	 \end{equation}
     where $1=\gamma\theta+(2-\delta)(1-\theta)$, $\theta=\frac{1-\delta}{2-\delta-\gamma}$ and $0<\delta<2-\gamma.$
    Thus
     \begin{equation}
     \begin{split}
      \|u\|_{C^1(\Omega)}&\leq C\|u_0\|_{C^1(\Omega)}+C_{T-\tau}H^{\frac{1-\delta}{2-\gamma-\delta}} \|\phi\|_{L^{\infty}(Q_{\tau,T})}^{\frac{1-\gamma}{2-\gamma-\delta}}.
     \end{split}
	 \end{equation}
\end{proof}
The following results are specifically designed for the case of one dimension.
\begin{lemma}\label{lem-diff}
  Let $\Omega = (0,L)$ for $L>0$, $0\leq \tau < T$. Let $f\in L^{\infty}(Q_{\tau,T})$ and $a:Q_{\tau,T}\to \R$ such that
	\blue{\begin{equation}\label{eq-a}
		0<\alpha_1 \leq a(x,t) \leq \alpha_1^{-1} \quad \forall (x,t)\in Q_{\tau,T},
	\end{equation}}
	\blue{for some $\alpha_1>0$}. Let $u$ be the solution of the parabolic equation
	\begin{equation}\label{eq-diff}
	\begin{cases}
		\frac{1}{a(x,t)}\pa_t u - \pa_{xx} u = f, & (x,t)\in Q_{\tau,T},\\
		\pa_xu(0,t) = \pa_xu(L,t) = 0, &t\in (\tau,T),\\
		u(x,\tau) = 0, &x\in \Omega.
	\end{cases}
	\end{equation}
	If $u\in L^{\infty}(Q_{\tau,T})$, $u\geq 0$, and $u_t\ge0$, then we have
	\begin{equation*}
		\|\pa_x u\|_{L^{\infty}(Q_{\tau,T})}\leq C,
	\end{equation*}
   where $C$ depends on  $T-\tau$, \blue{$\alpha_1$}, $\|f\|_{L^{\infty}(Q_{\tau,T})}$ and $\|u\|_{\LQtauT{\infty}}$.
\end{lemma}

\begin{proof}
We define $\bar{Q}_{\tau,T}=(-L,2L)\times (\tau,T)$, and $\bar u$, $\bar a$, $\bar f$: $\bar{Q}_{\tau,T} \to \R$ as follows
    \begin{equation}\label{eq-exten-u}
	\bar{u}(x,t)=\begin{cases}
		u(x,t),& x\in [0,L],\\
        u(2L-x,t),& x\in (L,2L],\\
        u(-x,t),& x\in [-L,0),
	\end{cases}
	\end{equation}
   \begin{equation}\label{eq-exten-a}
	\bar{a}=\begin{cases}
		a(x,t),& x\in [0,L],\\
        a(2L-x,t),& x\in (L,2L],\\
        a(-x,t),& x\in [-L,0)
	\end{cases}
	\end{equation}
and
    \begin{equation}\label{eq-exten-f}
	\bar{f}=\begin{cases}
		f(x,t),& x\in [0,L],\\
        f(2L-x,t),& x\in (L,2L],\\
        f(-x,t),& x\in [-L,0).
	\end{cases}
	\end{equation}
   Since $u$ solves \eqref{eq-diff}, \blue{it follows that $\bar{u}$ solves}
   \begin{equation}\label{eq-ext-w}
	\begin{cases}
		\frac{1}{\bar a}\pa_t \bar{u} - \pa_{xx} \bar{u} = \bar{f}, & (x,t)\in \bar{Q}_{\tau,T},\\
		\pa_x\bar{u}(-L,t) = \pa_x\bar{u}(2L,t) = 0, &t\in(\tau,T),\\
		\bar{u}(x,\tau) = 0, &x\in (-L,2L).
	\end{cases}
	\end{equation}
   Moreover,
   \begin{equation}\label{aa}
   \|\bar{f}\|_{L^{\infty}(\bar{Q}_{\tau,T})}=\|f\|_{L^{\infty}(Q_{\tau,T})}\quad \text{and }~ \|\bar{u}\|_{L^{\infty}(\bar{Q}_{\tau,T})}=\|u\|_{L^{\infty}(Q_{\tau,T})}.
   \end{equation}
\blue{Fix $\rho\in (0,L)$ and $x_0\in (0,L)$, and let} $\zeta=\zeta(|x-x_0|)$ with $\zeta\in C^{\infty}_0(\R)$ satisfying
\begin{equation*}\zeta(s)=
\begin{cases}
1,~&0\leq s\leq \frac{\rho}{2},\\
0,~&s\geq \rho>0,
\end{cases}
\end{equation*}
\blue{ $0\leq\zeta\leq 1$ and $|\zeta'|\leq \frac{1}{\rho}$.  By multiplying the pde in} \eqref{eq-ext-w} with $\bar{u}\zeta^2$ and integrating over $[x_0-\rho,x_0+\rho]\times(\tau,T)$, we obtain that
        \begin{equation*}
          \begin{split}
            \int^T_\tau\int^{x_0+\rho}_{x_0-\rho}&\frac{1}{\bar a}\bar{u}_t\bar{u}\zeta^2 dxdt+\int^T_\tau\int^{x_0+\rho}_{x_0-\rho}(\pa_{x}\bar{u})^2\zeta^2 dxdt\\
            &=\int^T_\tau\int^{x_0+\rho}_{x_0-\rho}-2 \bar{u}\pa_{x} \bar{u}\zeta\pa_x\zeta dxdt+\int^T_\tau\int^{x_0+\rho}_{x_0-\rho}\bar{f}\bar{u}\zeta^2 dxdt.
          \end{split}
	    \end{equation*}
	    The first term on the left hand side is non-negative due to the assumptions on $a$ and $u, u_t \geq 0$. By applying Cauchy's inequality $-2\bar{u}\pa_{x}\bar{u}\zeta\pa_x\zeta \leq \frac{1}{2}(\pa_{x}\bar{u})^2\zeta^2+2\bar{u}^2(\pa_x\zeta)^2$, we have
        \begin{equation}\label{eq-diff1}
          \begin{split}
            \int^T_\tau&\int^{x_0+\rho}_{x_0-\rho}(\pa_x\bar{u})^2\zeta^2 dxdt\\
            &\leq4\int^T_\tau\int^{x_0+\rho}_{x_0-\rho}\bar{u}^2(\pa_x\zeta)^2 dxdt+2\|\bar{f}\|_{L^{\infty}(\bar{Q}_T)}\|\bar{u}\|_{L^{\infty}(\bar{Q}_T)}
            \int^T_\tau\int^{x_0+\rho}_{x_0-\rho}\zeta^2dxdt\\
            &\leq\frac{8 (T-\tau)}{\rho}\|\bar{u}\|^2_{L^{\infty}(\bar{Q}_{\tau,T})}+4\rho T\|\bar{f}\|_{L^{\infty}(\bar{Q}_{\tau,T})}\|\bar{u}\|_{L^{\infty}(\bar{Q}_{\tau,T})}:=C_0,
          \end{split}
	    \end{equation}
        note  that, \blue{where \eqref{aa} implies} $C_0$ depends only on $\rho,~T-\tau,~\|f\|_{\LQtauT{\infty}}$ and $\|u\|_{\LQtauT{\infty}}$. Therefore
		\begin{equation}\label{eq-diff2}
          \begin{split}
            \int^T_\tau\int^{x_0+\frac{\rho}{2}}_{x_0-\frac{\rho}{2}}(\pa_x\bar{u})^2dxdt\leq C_0.
          \end{split}
	    \end{equation}

        \blue{Differenting equation \eqref{eq-ext-w} with respect to $x$ gives}
        \begin{equation*} 
	    \begin{cases}
		  \pa_t (\pa_x\bar{u}) - \pa_x(\bar{a}\pa_{xx} \bar{u}) =\pa_x(\bar{a}\bar{f}), & (x,t)\in \bar{Q}_{\tau,T},\\
		  \pa_x\bar{u}(-L,t) = \pa_x\bar{u}(2L,t) = 0, &t\in (\tau,T),\\
	 	   \pa_x\bar{u}(x,\tau) = 0, &x\in(-L,2L).
	   \end{cases}
	   \end{equation*}

        Let $\widehat{Q}_{\tau,T}=[x_0-\frac{\rho}{3},x_0+\frac{\rho}{3}]\times(\tau,T)$.  From \eqref{eq-diff2}, \eqref{eq-a} and the fact $\bar{a}\bar{f}\in L^{\infty}(\bar{Q}_{\tau,T})$,  by \cite[Chapter 3, Section 8, Theorem 8.1]{Ladyzenskaja1968Linear}, we have
        \begin{equation*}
          \begin{split}
            \|\pa_x \bar{u}\|_{L^{\infty}(\widehat{Q}_{\tau,T})}\leq C_1,
          \end{split}
	    \end{equation*}
         where $C_1$ depends on $C_0$ and \blue{the constant $\alpha_1$}, but is independent of $x_0$. Thus, since $x_0 \in (0,L)$ arbitrary and $\rho>0$, we can obtain
     \begin{equation*}
		\|\pa_x u\|_{\LQtauT{\infty}}\leq C_1.
	\end{equation*}
\end{proof}

\begin{lemma}[H\"older continuity in one dimension]\label{key_lem2}
	Let $\Omega = (0,L)$ for some $L>0$, and $0\leq \tau < T$. Assume that $f\in \LQtauT{\infty}$ and $a:Q_{\tau,T}\to \R$ such that
	\blue{\begin{equation*}
		0<\alpha_2 \leq a(x,t) \leq \alpha_2^{-1} \quad \forall (x,t)\in Q_{\tau,T},
	\end{equation*}}
	\blue{for some $\alpha_2>0$. Then}
	\begin{equation}\label{para}
	\begin{cases}
		\pa_t w - \pa_x(a\pa_x w ) = \pa_xf, & (x,t)\in Q_{\tau,T},\\
		w(0,t) = w(L,t) = 0, &t\in (\tau,T),\\
		w(x,\tau) = 0, &x\in \Omega
	\end{cases}
	\end{equation}
	has a unique weak solution. Moreover, if $w\in \LQtauT{\infty}$, then  w is H\"older continuous with an exponent $\delta\in(0,1)$, i.e.
	\begin{equation*}
		|w(x,t) - w(x',t')| \leq K_0\bra{|x - x'|^{\delta} + |t-t'|^{\delta/2}}, \quad \forall (x,t), (x',t') \in Q_{\tau,T},
	\end{equation*}
    where the constant $K_0$ depends on $\|w\|_{\LQtauT{\infty}},~\|f\|_{\LQtauT{\infty}}$ \blue{and $\alpha_2$}.
\end{lemma}

\begin{proof}
\blue{Similar to the proof} of Lemma \ref{lem-diff}, define $\bar{Q}_{\tau,T}=(-L,2L)\times (\tau,T)$, and set
   \blue{ \begin{equation}\label{eq-exten-w}
	\bar{w}(x,t)=\begin{cases}
		w(x,t),& x\in [0,L],\\
        -w(2L-x,t),& x\in (L,2L],\\
        -w(-x,t),& x\in [-L,0).
	\end{cases}
	\end{equation}}
  \blue{Then $\bar{w}$ is the weak solution of}
    \begin{equation}\label{eq-diff3}
	    \begin{cases}
		  \pa_t \bar{w} - \pa_x(\bar{a}\pa_{x} \bar{w}) =\pa_x\bar{f}, & (x,t)\in \bar{Q}_{\tau,T},\\
		  \bar{w}(-L,t) = \bar{w}(2L,t) = 0, &t\in (\tau,T),\\
	 	   \bar{w}(x,\tau) = 0, &x\in(-L,2L),
	   \end{cases}
	   \end{equation}
\blue{where $\bar{a}$ and $\bar{f}$ are given similarly to \eqref{eq-exten-w}}. Obviously,
   \[
   \|\bar{f}\|_{L^{\infty}(\bar{Q}_{\tau,T})}=\|f\|_{L^{\infty}(Q_{\tau,T})}\quad \text{and }~ \|\bar{w}\|_{L^{\infty}(\bar{Q}_{\tau,T})}=\|w\|_{L^{\infty}(Q_{\tau,T})}.
   \]

We fix $(x_0, t_0)\in Q_{\tau,T}$ and $0<\varrho<L$, denote
\[
B_\varrho = \{x\in \Omega: x\in(x_0-\varrho,x_0+\varrho)\}
\]
and, for $0<r\le t_0$,
\[
Q(\varrho,r) = B_\varrho \times (t_0-\tau,t_0) =  \{(x,t): x\in(x_0-\varrho,x_0+\varrho), t_0 - r < t < t_0\}.
\]
For any $0 < \tau_2 < \tau_1 < \min\{1,r\}$ and $0 < \varrho_2 < \varrho_1 < \varrho$ such that $Q(\varrho_1, \tau_1)\subset \bar{Q}_T$, \blue{
let $\xi: \bar{Q}_T \to [0,1]$ be} a smooth cut-off function such that
		\begin{equation*}
			\xi(x,t) = \begin{cases}
				1 &\text{ if } \quad (x,t) \in Q(\vr_2, \tau_2),\\
				0 &\text{ if } \quad (x,t) \not\in Q(\vr_1, \tau_1).
			\end{cases}
		\end{equation*}
For any $k>0$, we denote by $\bar{w}^{(k)} = (\bar{w}-k)_+$. \blue{By multiplying the pde in \eqref{eq-diff3} with $\bar{w}^{(k)}\xi^2$. Integrating over $B_{\vr_1}\times (t_0-\tau_1, t)$, $t_0 - \tau_2 < t < t_0$, and by integrating by parts, we calculate}
		\begin{equation}\label{mul}
			\begin{split}
				\int_{t_0-\tau_1}^t&\int_{B_{\vr_1}}\pa_t\bar{w} \bar{w}^{(k)}\xi^2dx ds + \int_{t_0-\tau_1}^t\int_{B_{\vr_1}}\bar{a}|\pa_x \bar{w}^{(k)}|^2\xi^2dxds\\
=& -2\int_{t_0-\tau_1}^t\int_{B_{\vr_1}} \bar{a}(\pa_x \bar{w}^{(k)}\cdot \pa_x \xi) \bar{w}^{(k)}\xi dxds -\int_{t_0-\tau_1}^t\int_{B_{\vr_1}}\bar{f}\pa_x\bar{w}^{(k)}\xi^2dxds\\
              &-2\int_{t_0-\tau_1}^t\int_{B_{\vr_1}}\bar{f}\bar{w}^{(k)}\xi\pa_x\xi dxds.
			\end{split}
		\end{equation}
		
		Since
		\begin{equation}\label{mul1}
			\begin{aligned}
			\int_{t_0-\tau_1}^t\int_{B_{\vr_1}}\pa_t\bar{w} \bar{w}^{(k)}\xi^2dx ds&= \frac{1}{2}\int_{B_{\vr_1}}(\bar{w}^{(k)}(t))^2\xi(t)^2dx - \int_{t_0-\tau_1}^t\int_{B_{\vr_1}}(\bar{w}^{(k)})^2\xi\pa_t\xi dxds\\
			&\geq \frac{1}{2}\int_{B_{\vr_2}}|\bar{w}^{(k)}(t)|^2dx - \int_{t_0-\tau_1}^t\int_{B_{\vr_1}}(\bar{w}^{(k)})^2\xi\pa_t\xi dxds\\
			&\geq \frac{1}{2}\int_{B_{\vr_2}}|\bar{w}^{(k)}(t)|^2dx - \int_{Q(\vr_1,\tau_1)}|\bar{w}^{(k)}|^2|\xi_t|dxds,
			\end{aligned}
		\end{equation}
		where we used $\xi(\cdot, t_0-\tau_1) = 0$ at the first step and $\xi|_{Q(\rho_2,\tau_2)} \equiv 1$ at the second step. \blue{By applying the Cauchy-Schwarz inequality}, we get
		\begin{equation}\label{mul2}
			\begin{split}
			-2\int_{t_0-\tau_1}^t&\int_{B_{\vr_1}}\bar{a}(\pa_x \bar{w}^{(k)}\cdot\pa_x\xi) \bar{w}^{(k)}\xi dxds \\
			&\leq  \frac 14 \int_{t_0-\tau_1}^t\int_{B_{\vr_1}}\bar{a}|\pa_x \bar{w}^{(k)}|^2\xi^2dxds+ 4\int_{t_0-\tau_1}^t\int_{B_{\vr_1}} \bar{a} |\bar{w}^{(k)}|^2|\pa_x  \xi|^2dxds\\
			&\leq  \frac 14 \int_{t_0-\tau_1}^t\int_{B_{\vr_1}}\bar{a}|\pa_x \bar{w}^{(k)}|^2\xi^2dxds+ 4\int_{Q(\vr_1,\tau_1)}\bar{a}|\bar{w}^{(k)}|^2|\pa_x \xi|^2dxds.
			\end{split}
		\end{equation}
Denoting by $\chi_{\{\bar{w} > k \}}$ the characteristic function of the set $\{(x,t)\in \bar{Q}_{\tau,T}: \bar{w}(x,t) > k\}$,
		we can estimate
		\begin{equation}\label{mul3}
			\begin{split}
			-&\int_{t_0-\tau_1}^t\int_{B_{\vr_1}}\bar{f}\pa_x \bar{w}^{(k)}\xi^2dxds-2\int_{t_0-\tau_1}^t\int_{B_{\vr_1}}\bar{f}\bar{w}^{(k)}\xi\pa_x\xi dxds\\
			&\leq \frac 14\int_{t_0-\tau_1}^t\int_{B_{\vr_1}}\bar{a}|\pa_x \bar{w}^{(k)}|^2\xi^2dxds + \|\bar{f}\|_{\LQtauT{\infty}}^2\int_{t_0-\tau_1}^t\int_{B_{\vr_1}}\frac{1}{\bar{a}}\xi^2\chi_{\{\bar{w}> k\}}dxds\\
            &\quad+2\|\bar{f}\|_{\LQtauT{\infty}}^2\int_{t_0-\tau_1}^t\int_{B_{\vr_1}}|\xi|^2\chi_{\{\bar{w} > k\}}dxds+\frac 12\int_{t_0-\tau_1}^t\int_{B_{\vr_1}}|\bar{w}^{(k)}|^2|\pa_x\xi|^2dxds\\
            &\leq \frac 14\int_{t_0-\tau_1}^t\int_{B_{\vr_1}}\bar{a}|\pa_x \bar{w}^{(k)}|^2\xi^2dxds + \frac{\|\bar{f}\|_{\LQtauT{\infty}}^2}{\blue{\alpha_2}}\int_{Q(\vr_1,\tau_1)}\chi_{\{\bar{w}> k\}}dxds\\
            &\quad +2\|\bar{f}\|_{\LQtauT{\infty}}^2\int_{Q(\vr_1,\tau_1)}\chi_{\{\bar w > k\}}dxds+\frac 12\int_{Q(\vr_1,\tau_1)}|\bar{w}^{(k)}|^2|\pa_x\xi|^2dxds\\
            &\leq \frac 14\int_{t_0-\tau_1}^t\int_{B_{\vr_1}}\bar{a}|\pa_x \bar{w}^{(k)}|^2\xi^2dxds + C_{T,\tau}\int_{Q(\varrho_1,\tau_1)}\bra{|w^{(k)}|^2|\pa_x\xi|^2 + \chi_{\{\bar w > k \}} }dxds
			\end{split}
		\end{equation}
		where
		\begin{equation*}
			C_{T,\tau}:= \max\left\{\frac{\|\bar f\|_{\LQtauT{\infty}}^2}{\blue{\alpha_2}}; \frac 12; 2\|\bar f\|_{\LQtauT{\infty}}^2 \right\}.
		\end{equation*}

We now insert the estimates \eqref{mul1}, \eqref{mul2} and \eqref{mul3} into \eqref{mul} to obtain that, for all $t\in (t_0-\tau_2, t_0)$,
		\begin{equation}\label{est}
		\begin{split}
            \frac{1}{2}&\|\bar{w}^{(k)}(t)\|_{L^2(B_{\vr_2})}^2 + \frac{\blue{\alpha_2}}{2}\int_{t_0-\tau_1}^t\int_{B_{\vr_1}}|\pa_x \bar{w}^{(k)}|^2\xi^2dxds\\
			&\leq  \int_{Q(\vr_1,\tau_1)}|\bar{w}^{(k)}|^2|\xi_t|dxds+\frac{4}{\blue{\alpha_2}}\int_{Q(\vr_1,\tau_1)}|\bar{w}^{(k)}|^2|\pa_x \xi|^2dxds\\
            &\quad + C_{T,\tau} \int_{Q(\vr_1,\tau_1)}\Big(|\bar{w}^{(k)}|^2|\pa_x\xi|^2+\chi_{\{\bar{w}> k\}}\Big)dxds\\
            &\leq C_{T,\tau} \int_{Q(\vr_1,\tau_1)}\Big(|\bar{w}^{(k)}|^2(|\xi_t|+|\pa_x \xi|^2) + \chi_{\{\bar{w}> k\}}\Big)dxds.
		\end{split}
		\end{equation}

By adding the inequality $\frac {\blue{\alpha_2}}{2} \int_{t_0-\tau_1}^{t}\int_{B_{\vr_1}}|\bar{w}^{(k)}|^2\xi^2dxds \leq \frac {\blue{\alpha_2}}{2} \int_{Q(\vr_1,\tau_1)}|\bar{w}^{(k)}|^2dxds$ on both sides of \eqref{est}, and taking the supremum over $t\in (t_0 - \tau_2, t_0)$, we have that
		\begin{equation}\label{est1}
			\begin{split}
			\frac{1}{2}&\left(\sup_{t_0 - \tau_2 < t < t_0}\|\bar{w}^{(k)}(t)\|_{L^2(B_{\vr_2})}^2 + \int_{t_0 - \tau_2}^{t_0}\|\bar{w}^{(k)}\|_{H^1(B_{\vr_2})}^2ds\right)\\
			&\leq C_{T,\tau} \int_{Q(\vr_1,\tau_1)}\Big(|\bar{w}^{(k)}|^2(|\xi_t|+|\pa_x \xi|^2) +\chi_{\{\bar{w}> k\}}\Big)dxds.
			\end{split}
		\end{equation}
		
Finally, due to the definition of the cut-off function $\xi$, there exists a constant $C\geq 1$ independent of $\vr_i$ and $\tau_i$ such that $|\pa_x \xi| \leq C(\vr_1-\vr_2)^{-1}$ and $|\pa_t \xi| \le C(\tau_1 - \tau_2)^{-1}$. Noting also $1 \leq (\tau_1 - \tau_2)^{-1}$ since $\tau_1, \tau_2 \in (0,1)$, we get from \eqref{est1} the energy estimate
\begin{equation}\label{local-en}
	\begin{split}
	&\sup_{t_0-\tau_2 < t < t_0}\|(\bar{w}-k)_+\|_{L^2(B_{\varrho_2})}^2+ \int_{t_0-\tau_2}^{t_0}\|(\bar{w}-k)_+\|_{H^1(B_{\varrho_2})}^2ds\\
	&~~\leq C_{T,\tau}\left[((\varrho_1-\rho_2)^{-2} + (\tau_1 - \tau_2)^{-1})\|(\bar{w} -k)_+\|_{L^2(Q(\varrho_1,\tau_1))}^2 + \int_{Q(\varrho_1,\tau_1)}\chi_{\{\bar{w}>k\}}dxds\right],
\end{split}
\end{equation}
where the constant $C$ depends only on $\|f\|_{\LQtauT{\infty}}$ and $\alpha$, but is independent of $x_0$ and $t_0$.

\blue{By the arbitrariness of} $x_0\in (0,L)$ and $t_0\in (\tau,T)$, combining with the estimates \eqref{local-en} and the fact $\|w\|_{L^\infty(Q_{\tau,T})} \leq C_{T,\tau}$, thanks to \cite[\blue{Theorem II.7.1, Lemmas 7.2 and 7.3}]{Ladyzenskaja1968Linear}, we can obtain the local H\"older continuity of $\bar{w}$ on $(-\varrho_2,L+\varrho_2)\times (\tau,T)$, which implies the global H\"older continuity of $w$ on $Q_{\tau,T}$ immediately, that is, there exist constant $\delta\in (0,1)$ such that
	\begin{equation*}
		|w(x,t) -w(x', t')| \leq C(|x - x'|^{\delta} + |t - t'|^{\delta/2})  \quad \text{for all} \quad (x,t), (x', t') \in Q_{\tau,T},
	\end{equation*}
where the constant $C$ depends only on $\|f\|_{\LQtauT{\infty}}$, $\|w\|_{L^\infty(Q_{\tau,T})}$ and $\blue{\alpha_2}$.
\end{proof}

\subsection{Global existence}
\medskip
The following bound in $L^\infty(0,T;L^1(\Omega))$ is immediate.
\begin{lemma}\label{L1-bound}
	\blue{Assume \eqref{A1} and \eqref{A2}.} For any $0<T<T_{\max}$, it holds
	\begin{equation*}
		\sup_{i=1,\ldots, m}\|u_i(t)\|_{\LO{1}} \leq C_T \quad \text{ for all } \quad t\in (0,T).
	\end{equation*}
\end{lemma}
\begin{proof}
	By summing the equations in \eqref{eq1.1}, integrating on $\Omega$ and using \eqref{A2}, we have
	\begin{equation*}
		\frac{d}{dt}\sumi \intO u_i(x,t)dx \leq k_0 + k_1\sumi \intO u_i(x,t)dx.
	\end{equation*}
	The classical Gronwall inequality gives the desired estimate.
\end{proof}

Summing the equations of \eqref{eq1.1}, it follows from \eqref{A2'} that
\begin{equation}\label{ee}
	\sumi u_i(x,t) - \pa_{xx} \int_0^t\sumi d_iu_i(x,s)ds = \sumi u_{i,0}(x) + \int^t_0g(x,s)ds =: G(x,t).
\end{equation}
Denote by
\begin{equation}\label{def_v}
	v(x,t):= \int_0^t\sumi d_iu_i(x,s)ds.
\end{equation}
Then we have
\begin{equation}\label{eq_v1}
	\begin{cases}
		b(x,t)\pa_tv(x,t) - \pa_{xx} v(x,t) = G(x,t),		 &x\in\Omega, \; t>0,\\
		\pa_xv(0,t) = \pa_xv(L,t) = 0, &t>0,\\
		v(x,0) = 0, &x\in\Omega,
	\end{cases}
\end{equation}
where
\begin{equation}\label{def_b}
	\min_{i=1,\ldots, m}\left\{\frac{1}{d_i} \right\}\leq b(x,t):= \frac{\sumi\limits u_i(x,t)}{\sumi\limits d_iu_i(x,t)} \leq \max_{i=1,\ldots, m}\left\{\frac{1}{d_i} \right\}.
\end{equation}
\blue{It follows from \eqref{ee} that $v$ is a solution to}
\begin{equation}\label{eq_v2}
	\begin{cases}
		\pa_tv(x,t) - \pa_{xx} v(x,t) = U_d(x,t) + G(x,t), &x\in\Omega, \; t>0,\\
		\pa_x v(0,t) = \pa_xv(L,t) = 0, &t>0,\\
		v(x,0) = 0, &x\in\Omega,
	\end{cases}
\end{equation}
where
\begin{equation}\label{def_Ud}
	U_d(x,t):= \sumi (d_i - 1)u_i(x,t).
\end{equation}

The following lemma is crucial as it shows that, in one dimension, also the spatial derivative of $v$ is H\"older continuous.

\begin{lemma}\label{lem3}
	The function $v$ defined in \eqref{def_v} satisfies
\blue{\begin{equation*}
|\pa_xv(x,t) - \pa_xv(x',t)| \leq \mathcal C_0|x'-x|^{\alpha} \quad \forall (x',t),(x,t)\in Q_T,
	\end{equation*}
where $\mathcal C_0$ and $\alpha\in (0,1)$ depend only on $T$, $\|G\|_{\LQ{\infty}}$, $L$, and $d_i$, $0<T<T_{\max}$.}
\end{lemma}
\begin{proof}
	We first claim that
   \begin{equation}\label{v_x_bounded}
   \|\pa_xv(x,t)\|_{\LQ{\infty}} \le C_T.
   \end{equation}
   \blue{To prove this we first show}
   \begin{equation}\label{Linfty-v}
   	\|v\|_{\LQ{\infty}} \leq C_T.
   \end{equation}
   \blue{This was in fact showed in \cite{fellner2020global}, but in one dimension, it can be shown by elementary arguments}. \blue{Note also that $C_T$ in \eqref{v_x_bounded} and \eqref{Linfty-v} depend only on $T$, $\|G\|_{\LQ{\infty}}$, $L$, and $d_i$.} Indeed, thanks to Lemma \ref{L1-bound} we have
   \begin{equation}\label{L1-v}
   	\|v(t)\|_{\LO{1}} \leq \sumi d_i\int_0^t\|u_i(s)\|_{\LO{1}}ds \leq C_T \quad \text{ for all } t\in (0,T).
   \end{equation}
   \blue{Multiplying both sides of \eqref{ee} by $v$,  integrating over $\Omega$,} and using $u_i \geq 0$ and $v\geq 0$, we have
   \begin{equation*}
	   	\intO|\pa_x v(x,t)|^2dx \leq \intO G(x,t)v(x,t)dx \leq \frac 14\|G(t)\|_{\LO{2}}^2 + \|v(t)\|_{\LO{2}}^2.
   \end{equation*}
   Adding $\|v(t)\|_{\LO{2}}^2$ to both sides gives
   \begin{equation*}
   		\|v\|_{H^1(\Omega)}^2 \leq  \frac{|\Omega|}{4}\|G\|_{\LQ{\infty}}^2 + 2\|v(t)\|_{\LO{2}}^2.
   \end{equation*}
	The interpolation inequality $\|v(t)\|_{L^2(\Omega)}  \leq C\|v(t)\|_{H^1(\Omega)}^{1/3}\|v(t)\|_{L^1(\Omega)}^{2/3}$ and Cauchy-Schwarz's inequality lead to
	\begin{equation*}
		\blue{\|v\|_{H^1(\Omega)} \leq C\|G\|_{\LQ{\infty}} \leq C_T.}
	\end{equation*}
	Now we use the one dimensional embedding $H^1(\Omega)\hookrightarrow L^\infty(\Omega)$ to eventually obtain the estimate \eqref{Linfty-v}.

	\medskip
	From \eqref{Linfty-v} and the fact that $v\geq 0$, $\pa_t v\geq 0$, we can apply Lemma \ref{lem-diff} to obtain the boundedness of the gradient \eqref{v_x_bounded}.
	
	\medskip
	Define $a(x,t):= (b(x,t))^{-1}$ we get from \eqref{eq_v1} that
	\begin{equation*}
		\pa_tv(x,t) - a(x,t)\pa_{xx}v(x,t) = a(x,t)G(x,t).
	\end{equation*}	
    By differentiating this equation with respect to $x$, and denoting $w(x,t):= \pa_x v(x,t)$, we obtain
	\begin{equation*}
		\pa_t w(x,t) - \pa_x(a(x,t)\pa_x w(x,t)) = \pa_x(a(x,t)G(x,t)).
	\end{equation*}
	Note that $w$ satisfies homogeneous Dirichlet boundary condition since $w(0,t) = \pa_xv(0,t) = 0$, \blue{and  $w(L,t) = \pa_xv(L,t) = 0$.}

Now we can apply Lemma \ref{key_lem2} to obtain the H\"older continuity of $w$, which finishes the proof of Lemma \ref{lem3}.
\end{proof}

\begin{lemma}\label{lem4}
	\blue{Assume \eqref{A1} and \eqref{A2}.} Let $v$ be the solution to \eqref{eq_v2}. It holds, for any $0<\delta < 2- \alpha$,
	\begin{equation*}
		\sup_{Q_T}|\pa_{xx}v| \leq \mathcal{C}_1\mathcal{C}_0^{\frac{1-\delta}{2-\alpha-\delta}}\left[\sup_{Q_T}|\pa_x U_d + \pa_x G| \right]^{\frac{1-\alpha}{2-\alpha-\delta}}
	\end{equation*}
	where the constants $\mathcal{C}_0$ and $\alpha$ are given in Lemma \ref{lem3},  and \blue{$\mathcal{C}_1=\mathcal{C}_1(T, \mathcal{C}_0, \alpha, \delta$)}.
\end{lemma}
\begin{proof}
	By differentiating \eqref{eq_v2} with respect to $x$, we have
	\begin{equation*}
		\pa_t(\pa_x v) - \pa_{xx}(\pa_xv(x,t)) = \pa_x\bra{U_d(x,t) + G(x,t)},
	\end{equation*}
	and $\pa_xv$ satisfies homogeneous Dirichlet boundary condition $\pa_xv(0,t) = \pa_xv(L,t) = 0$. Thanks to the H\"older continuity of $\pa_xv$ in Lemma \ref{lem3}, we can apply Lemma \ref{lem-Diri} to conclude the proof of Lemma \ref{lem4}.
\end{proof}

We are now ready to show the global existence part of Theorem \ref{th1.2}.

\begin{proof}[\textbf{Proof of Theorem \ref{th1.2} - Global existence}]
	We denote by
	\begin{equation*}
		U:= \sup_{Q_T}\sup_{i}|u_i(x,t)|.
	\end{equation*}
	From the equation of $u_i$
	\begin{equation*}
		\pa_t u_i - d_i\pa_{xx}u_i = f_i(x,t,u),
	\end{equation*}
	assumption \eqref{A3}, and Lemma \ref{lem2.3}, we have
	\begin{equation*}
		\sup_{Q_T}|\pa_x u_i| \leq \sup_{\Omega}|\pa_xu_{i,0}| + CU^{\frac 12}\left[1+U^{3+\varepsilon}\right]^{\frac 12} \leq C\bra{1+U^{2+\frac \varepsilon2}}.
	\end{equation*}
	From \eqref{def_Ud}, we know
  \begin{equation*}
	\sup_{Q_T}|\pa_x U_d|\leq \sumi (d_i - 1)\sup_{Q_T}|\pa_x u_i|.
  \end{equation*}
  It follows that
	\begin{equation*}
		\sup_{Q_T}|\pa_x U_d| \leq C\bra{1+U^{2+\frac \varepsilon2}}.
	\end{equation*}	
	Therefore, Lemma \ref{lem4} implies
	\begin{equation*}
		\sup_{Q_T}|\pa_{xx}v| \leq C_T\left[1+U^{2+\frac \varepsilon2}\right]^{\frac{1-\alpha}{2-\alpha-\delta}} \leq C_T\left[1+U^{(2+\frac \varepsilon2)\frac{1-\alpha}{2-\alpha-\delta}}\right].
	\end{equation*}
	It then follows from \eqref{ee} that
	\begin{equation*}
		U \leq \sup_{Q_T}|\pa_{xx}v| + \|G\|_{\LQ{\infty}} \leq C_T\left[1 + U^{(2+\frac \varepsilon2)\frac{1-\alpha}{2-\alpha-\delta}}\right].
	\end{equation*}
    We can choose $\varepsilon$ and $\delta$ small enough such that
    \begin{equation*}
       (2+\frac{\varepsilon}{2})\frac{1-\alpha}{2-\alpha-\delta}<1 \quad\text{or equivalently}\quad  \varepsilon<\frac{2(\alpha-\delta)}{1-\alpha}.
    \end{equation*}
	By apply Young's inequality to finally obtain
	\begin{equation*}
		U \leq C_T,
	\end{equation*}
	which confirms the global existence of \eqref{eq1.1}.	
\end{proof}

\subsection{Uniform-in-time bounds}
Now we consider the case where $g\equiv 0$ in the assumption \eqref{A2'}. To obtain the uniform-in-time bound for this solution,  we just need to show that
\begin{equation*}
     \sup_{\tau\in \mathbb N}\|u_i\|_{L^\infty(Q_{\tau,\tau+1})}\leq C.
\end{equation*}
where we recall $Q_{\tau, \tau+1} = \Omega\times(\tau,\tau+1)$.

\medskip
Let $\tau \in \mathbb N$ and $\vpt: \R_+ \to [0,1]$ a smooth function such that $\vpt(s) = 0$ for $s\leq \tau$, $\vpt(s) = 1$ for $s\ge \tau+1$, and $\vpt'(s) \geq 0$ for all $s\in \R_+$. By multiplying the equation \eqref{eq1.1} by $\vpt$ and denoting
\begin{equation}\label{def_w}
	w_i(x,t) = \vpt(t)u_i(x,t),
\end{equation}
 we obtain
\begin{equation}\label{eq_w_tau}
	\begin{cases}
		\pa_t w_i - d_i\pa_{xx}w_i = \vpt' u_i + \vpt f_i(x,t,u), &(x,t)\in Q_{\tau,\tau+2},\\
		\pa_xw_i(0,t) = \pa_xw_i(L,t) = 0, &t\in (\tau,\tau+2),\\
		w_i(x,\tau) = 0, &x\in\Omega.
	\end{cases}
\end{equation}

We sum the equations of \eqref{eq_w_tau}, use $g\equiv 0$, and integrate the result on $(\tau,t)$, $t\in(\tau,\tau+2)$ to obtain
\begin{equation}\label{eq_v_tau1}
	\sumi w_i(x,t) - \pa_{xx}\int_{\tau}^{t}\sumi d_iw_i(x,s)ds = z(x,t):= \sumi \int_{\tau}^{t}\vpt' u_i(x,s)ds.
\end{equation}
With
\begin{equation}\label{def_v_tau}
	v(x,t):= \int_{\tau}^{t}\sumi d_iw_i(x,s)ds
\end{equation}
and the function $b(x,t)$ defined as in \eqref{def_b} we also have
\begin{equation}\label{eq_v_tau2}
	\begin{cases}
		b(x,t)\pa_t v - \pa_{xx}v = z(x,t), &(x,t)\in Q_{\tau,\tau+2},\\
		\pa_xv(0,t) = \pa_xv(L,t) = 0, &t\in (\tau,\tau+2),\\
		v(x,\tau) = 0, &x\in\Omega
	\end{cases}
\end{equation}
and
\begin{equation}\label{eq_v_tau3}
	\begin{cases}
		\pa_t v - \pa_{xx}v = \Phi(x,t):= \sumi (d_i-1)w_i(x,t) + z(x,t), &(x,t)\in Q_{\tau,\tau+2},\\
		\pa_xv(0,t) = \pa_xv(L,t) = 0, &t\in(\tau,\tau+2),\\
		v(x,\tau) = 0, &x\in\Omega.
	\end{cases}
\end{equation}

\begin{lemma}[Lemma 2.1, \cite{fellner2021uniform}]\label{from_fmt}
	Let $\Omega\subset\R^n$ be bounded with smooth boundary, let $d>0$ and $0\leq \tau < T$. Let $u$ be the solution to
	\begin{equation*}
	\begin{cases}
		\pa_t u - d\Delta u = \psi(x,t), &(x,t)\in Q_{\tau,T},\\
		\nabla u\cdot \nu = 0, &(x,t)\in\pa\Omega\times(\tau,T),\\
		u(x,\tau) = u_0(x), &x\in\Omega.
	\end{cases}
	\end{equation*}
	Assume that there exists $\gamma \in [0,1)$ such that
	\begin{equation*}
		\blue{|u(x,t) - u(x',t)| \leq H|x - x'|^{\gamma} \quad \forall (x,t), (x',t)\in Q_{\tau,T}.}
	\end{equation*}
	Then
	\begin{equation*}
		\sup_{Q_{\tau,T}}|\nabla u(x,t)| \leq C\|u_0\|_{C^1(\Omega)} + BH^{\frac{1}{2-\gamma}}\|\psi\|_{L^{\infty}(Q_{\tau,T})}^{\frac{1-\gamma}{2-\gamma}}
	\end{equation*}
	where $B>0$ and $C>0$ are constants depending only on $\Omega$, $n$, $d$ and $\gamma$.
\end{lemma}
\begin{remark}
  \blue{It is remarked that Lemma \ref{from_fmt} looks similar to Lemma \ref{lem2.3}, except for the fact that it is considered in the cylinder $\Omega\times(\tau,T)$. Due to $t$-dependence of $\psi$, the problem is non-autonomous. Nevertheless, the proof is in fact similar to that of Lemma \ref{lem2.3}.}
\end{remark}

\begin{lemma}\label{new_lem1}
	\blue{Assume \eqref{A1} and \eqref{A2}.} There exists a constant $C>0$ such that
	\begin{equation*}
		\sup_{\tau\in\N} \left(\|v\|_{L^{\infty}(Q_{\tau,\tau+2})} + \|z\|_{L^\infty(Q_{\tau,\tau+2})}\right) \leq C.
	\end{equation*}	
\end{lemma}
\begin{proof}
	The bound of $v$ was proved in \cite[Lemma 3.3]{fellner2021uniform}. The bound of $z$ follows immediately due to its definition in \eqref{eq_v_tau1} and the bound of $v$.
\end{proof}

\begin{lemma}\label{new_lem2}
	\blue{Assume \eqref{A1} and \eqref{A2}.}  There exists a constant $C>0$ such that
	\begin{equation}\label{v_x_tau_bounded}
		\sup_{\tau\in \N}\|\pa_x v\|_{L^{\infty}(Q_{\tau,\tau+2})} \leq C.
	\end{equation}
\end{lemma}
\begin{proof}
	To prove this lemma, we apply Lemma \ref{lem-diff} to the equation \eqref{eq_v_tau2} on the interval $(\tau,\tau+2)$. Then, it holds
	\begin{equation*}
		\|\pa_x v\|_{L^{\infty}(Q_{\tau,\tau+2})} \leq \mathscr{C}
	\end{equation*}
	where the constant $\mathscr{C}$ depends only on $d_i$ and the bounds $\|z\|_{L^{\infty}(Q_{\tau,\tau+2})}$ and $\|v\|_{L^{\infty}(Q_{\tau,\tau+2})}$. Since these last two quantities are independent of $\tau\in \N$, the desired estimate \eqref{v_x_tau_bounded} follows.
\end{proof}

\begin{lemma}\label{new_lem3}
	\blue{Assume \eqref{A1} and \eqref{A2}.}  There exist $\beta \in (0,1)$ and $H>0$, which is {\normalfont independent of $\tau\in \N$}, such that
	\begin{equation*}
		|\pa_x v(x,t) - \pa_xv(x',t)| \leq H|x-x'|^{\beta}, \quad \forall (x,t), (x',t)\in Q_{\tau,\tau+2}.
	\end{equation*}
\end{lemma}
\begin{proof}
	By dividing both sides of \eqref{eq_v_tau2} by $b(x,t)$ then differentiating with respect to $x$, we obtain
	\begin{equation*}
		\begin{cases}
			\pa_t (\pa_x v) - \pa_x(b(x,t)^{-1}\pa_x(\pa_xv)) = \pa_x(b(x,t)^{-1}z(x,t)), &(x,t)\in Q_{\tau,\tau+2},\\
			\pa_xv(0,t) = \pa_xv(L,t) = 0, &t\in (\tau,\tau+2),\\
			\pa_xv(x,0) = 0, &x\in\Omega.
		\end{cases}
	\end{equation*}
	Now we can apply Lemma \ref{key_lem2} to this equation (with $w$ is equal to $\pa_x v$) to obtain the H\"older continuity
	\begin{equation*}
		|\pa_x v(x,t) - \pa_xv(x',t)| \leq H|x-x'|^{\beta}, \quad \forall (x,t), (x',t)\in Q_{\tau,\tau+2}
	\end{equation*}
	where we notice that the constant $H$ depends on $\|\pa_xv\|_{L^{\infty}(Q_{\tau,\tau+2})}$ and $\|b^{-1}z\|_{L^{\infty}(Q_{\tau,\tau+2})}$, which are in turn \textit{independent of $\tau\in \N$}, thanks to Lemmas \ref{new_lem1} and \ref{new_lem2}.
\end{proof}

We also need the following elementary lemma, whose proof is straightforward.
\begin{lemma}\label{lem2.9}
  Let $\{y_n\}_{n\geq0}$ be a nonnegative sequence. Define $\mathcal{N}=\{n\in \mathbb{N} :y_{n-1}\leq y_n\}.$ If there exists $c_0>0$ such that
  $$y_n\leq c_0 \quad\text{for all}\quad n\in\mathcal{ N},$$
  then $$y_n\leq c=\max\{y_0,c_0\},\quad \text{for all} \quad n\in \mathbb{N},$$
  where a constant $c$ independent of $n$.
\end{lemma}

\begin{proof}[\textbf{Proof of Theorem \ref{th1.2} - Uniform-in-time bounds}]
	By differentiating \eqref{eq_v_tau3} with respect to $x$ we have
	\begin{equation}\label{eq_v_tau4}
	\begin{cases}
		\pa_t(\pa_x v) - \pa_{xx}(\pa_xv) = \pa_x \Phi, &(x,t)\in Q_{\tau,\tau+2},\\
		\pa_xv(0,t) = \pa_x(L,t) = 0, &t\in (\tau,\tau+2),\\
		\pa_xv(x,\tau) = 0, &x\in\Omega.
	\end{cases}
	\end{equation}
	\blue{By applying Lemmas} \ref{new_lem3} and \ref{lem-Diri} to \eqref{eq_v_tau4} we have, for any $0<\delta<2-\beta$,
	\begin{equation}\label{new1}
		\begin{aligned}
		\sup_{Q_{\tau,\tau+2}}|\pa_{xx}v| & \leq CH^{\frac{1-\delta}{2-\beta-\delta}}\left[\sup_{Q_{\tau,\tau+2}}|\pa_x \Phi|\right]^{\frac{1-\beta}{2-\beta-\delta}}\\
		&\leq C\left[ \sup_{Q_{\tau,\tau+2}}\sumi|\pa_x w_i| + |\pa_x z| \right]^{\frac{1-\beta}{2-\beta-\delta}}.
		\end{aligned}
	\end{equation}
	Define
	\begin{equation*}
		U:= \sup_{Q_{\tau,\tau+2}}\sup_{i=1,\ldots, m}|u_i(x,t)|.
	\end{equation*}	
	By applying Lemma \ref{from_fmt} to the equation \eqref{eq_w_tau} and using the growth \eqref{A3}, we have
	\begin{equation}\label{new2}
		\begin{aligned}
			\sup_{Q_{\tau,\tau+2}}|\pa_x w_i| &\leq C\sup_{Q_{\tau,\tau+2}}|w_i|^{1/2}\left[\sup_{Q_{\tau,\tau+2}}\bra{|\vpt' u_i| + |\vpt f_i(x,t,u)|} \right]^{1/2}\\
			&\le CU^{1/2}\left[U + 1 + U^{3+\eps}\right]^{1/2}\\
			&\leq C\sbra{1 + U^{2+\frac\eps 2}}.
		\end{aligned}
	\end{equation}
	By integrating \eqref{eq1.1} on $(\tau,t)$ and denoting $y_i(x,t) = \int_{\tau}^tu_i(x,s)$ we have
	\begin{equation}\label{new3}
	\begin{cases}
		\pa_t y_i - d_i\pa_{xx}y_i = \int_{\tau}^{t}f_i(x,s,u)ds, &(x,t)\in Q_{\tau,\tau+2},\\
		\pa_xy_i(0,t) = \pa_xy_i(L,t) = 0, &t\in (\tau,\tau+2),\\
		y_i(x,\tau) = 0, &x\in\Omega.
	\end{cases}
	\end{equation}
	Using \eqref{A3} and Lemma \ref{new_lem1} we can estimate
	\begin{align*}
		\sup_{(x,t)\in Q_{\tau,\tau+2}}\left|\int_{\tau}^{t}f_i(x,s,u)ds \right| &\leq C\sup_{Q_{\tau,\tau+2}}\int_{\tau}^{\tau}\sbra{1+\sumi u_i(x,s)}\sbra{1 + \sumi u_{i}(x,s)}^{2+\eps}ds\\
		&\leq C\bra{1 + U}^{2+\eps}\sup_{Q_{\tau,\tau+2}}\int_{\tau}^{\tau+2}\sbra{1+\sumi u_i(x,s)}ds\\
		&\leq C\bra{1+U}^{2+\eps}.
	\end{align*}
	Now, we can apply Lemma \ref{from_fmt} to \eqref{new3} to have
	\begin{align}\label{new4}
		\sup_{Q_{\tau,\tau+2}}|\pa_x y_i| \leq C\sbra{\sup_{Q_{\tau,\tau+2}}|y_i|}^{1/2}\sbra{\sup_{Q_{\tau,\tau+2}}\left|\int_{\tau}^t f_i(x,s,u)ds\right|}^{1/2} \leq C\bra{1+U}^{1+\frac{\eps}{2}}
	\end{align}
	where we used $\sup_{Q_{\tau,\tau+2}}|y_i| \leq C$ at the last step thanks to Lemma \ref{new_lem1}. From the definition of $z$ in \eqref{eq_v_tau1} and \eqref{new4} it follows that
	\begin{equation*}
		\sup_{Q_{\tau,\tau+2}}|\pa_x z| \leq C\sbra{1+U^{1+\frac{\eps}{2}}}.
	\end{equation*}
	Inserting this and \eqref{new2} into \eqref{new1} yields
	\begin{equation*}
		\sup_{Q_{\tau,\tau+2}}|\pa_{xx}v| \leq 	C\sbra{1 + U^{2+\frac{\eps}{2}}}^{\frac{1-\beta}{2-\beta-\delta}} \leq C\sbra{1 + U^{\bra{2+\frac\eps 2}\frac{1-\beta}{2-\beta-\delta}}}.
	\end{equation*}
	From this and \eqref{eq_v_tau1}, it implies
	\begin{equation*}
		\sup_{Q_{\tau,\tau+2}}\sup_{i=1,\ldots, m}|w_i(x,t)| \leq \sup_{Q_{\tau,\tau+2}}|\pa_{xx}v| + C\sup_{Q_{\tau,\tau+2}}|y_i| \le C\sbra{1 + U^{\bra{2+\frac\eps 2}\frac{1-\beta}{2-\beta-\delta}}}.
	\end{equation*}
	We consider the set
	\begin{equation*}
		\mathscr{N}:= \left\{\tau \in \N:\,
		\sup_{Q_{\tau,\tau+1}}\sup_{i=1,\ldots, m}|u_i(x,t)| \leq \sup_{Q_{\tau+1,\tau+2}}\sup_{i=1,\ldots, m}|u_i(x,t)|\right\}.
	\end{equation*}
	Now for all $\tau\in \mathscr{N}$, we use $\vpt \geq 0$ and $w_i(x,t) = u_i(x,t)$ for all $(x,t)\in Q_{\tau+1,\tau+2}$ to have
	\begin{equation}\label{new5}
		U \leq 2\sup_{Q_{\tau+1,\tau+2}}\sup_{i=1,\ldots, m}|u_i(x,t)| \leq C\sbra{1 + U^{\bra{2+\frac\eps 2}\frac{1-\beta}{2-\beta-\delta}}}.
	\end{equation}
	Therefore, for $\eps>0$ and $\delta>0$ small enough,
	\begin{equation*}
		\bra{2+\frac\eps 2}\frac{1-\beta}{2-\beta-\delta} < 1,
	\end{equation*}
	and consequently, \eqref{new5} and Young's inequality imply
	\begin{equation}\label{new6}
		\sup_{Q_{\tau,\tau+2}}\sup_{i=1,\ldots, m}|u_i(x,t)| = U \leq C \quad \forall \tau\in \mathscr{N},
	\end{equation}
	where $C$ is \textit{independent of $\tau$}. Finally, by applying Lemma \ref{lem2.9}, the estimate \eqref{new6} is true for all $\tau\in \N$, and the proof of uniform-in-time bounds of Theorem \ref{th1.2} is finished.
\end{proof}

\section{Proof of Theorem \ref{th1.1}}\label{sub2.2}
\begin{proof}[Proof of Theorem \ref{th1.1}]
Similar to Theorem \ref{th1.2}, the global existence follows if one can show that
\begin{equation*}
	\sup_{i=1,\ldots, m}\sup_{T\in (0,T_{\max})}\|u_i\|_{\LQ{\infty}} < +\infty,
\end{equation*}
where $(0,T_{\max})$ is the maximal interval of the local strong solution.

\medskip
\blue{Assume for contradiction that $T_{\max}<\infty$.
We follow the idea from \cite{fellner2020global} which states that} with a suitable change of unknowns, and especially adding one more appropriate equation, we can transform a system with the mass control condition \eqref{A2} into a system with condition \eqref{A2'}, that keeps the essential features \eqref{A1} and \eqref{A3}.

\medskip
In order to do that, we define
\begin{equation*}
  \begin{split}
    w_i(x,t)=e^{-k_1t}u_i(x,t) \quad \text{or equivalently}  \quad u_i(x,t)=e^{k_1t}w_i(x,t).
  \end{split}
\end{equation*}

Direct computations give
\begin{equation*}
  \begin{split}
    \partial_tw_i&=e^{-k_1t}(\partial_tu_i-k_1u_i)\\
    &=e^{-k_1t}(d_i\pa_{xx} u_i+f_i(x,t,u)-k_1u_i)\\
    &=d_i\pa_{xx} w_i+g_i(x,t,w),
  \end{split}
\end{equation*}
where
\begin{equation}\label{eq-def-g}
  g_i(x,t,w)=e^{-k_1t}(f_i(x,t,u)-k_1e^{k_1t}w_i).
\end{equation}

Note that
\begin{equation*}
  \begin{split}
    \sum^m_{i=1}g_i(x,t,w)=e^{-k_1t}\sum^m_{i=1}(f_i(x,t,u)-k_1u_i)\leq e^{-k_1t}k_0
  \end{split}
\end{equation*}
due to the assumption \eqref{A2}. Introduce a new unknown $w_{m+1}:\Omega\times(0, T_{\max})\rightarrow\mathbb{R}_{+}$ which solves
\begin{equation}\label{eq-m+1}
  \begin{split}
    \partial_tw_{m+1}-\pa_{xx} w_{m+1}=k_0e^{-k_1t}-\sum^m_{i=1}g_i(x,t,w)=:g_{m+1}(x,t,w)\geq0
  \end{split}
\end{equation}
with homogeneous Neumann boundary condition $\pa_x w_{m+1}(0,t) = \pa_x w_{m+1}(L,t)=0$ and zero initial data $w_{m+1}(x,0)=w_{m+1,0}(x)=0$ for $x\in\Omega$. With a slight abuse of notation we write the new vector of concentrations $\widetilde{w}=(w_1, \ldots, w_m, w_{m+1})$ and the nonlinearities $g_i(x,t,\widetilde{w}):=g_i(x,t,w_1, \ldots, w_m)$ for all $i=1,\ldots, m$ while $g_{m+1}(x,t,\widetilde{w})=k_0e^{-k_1t}-\sum^m_{i=1}g_i(x,t,w)$. We have arrived at the following system

\begin{equation}\label{eq-w}
  \begin{cases}
    \partial_{t}w_i-d_i\partial_{xx}w_i=g_i(x,t,\widetilde{w}), &(x,t)\in \Omega\times(0,T_{\max}),~ i=1,2,\cdots,m+1,\\
    \partial_xw_i(0,t)=\partial_xw_i(L,t)=0,    &t>0,\;i=1,2,\cdots,m+1,\\
    w_i(x,0)=u_{i,0}(x),&x\in\Omega,~i=1,2,\cdots,m,\\
    w_{m+1}(x,0)=w_{m+1,0}(x)=0, &x\in\Omega,
  \end{cases}
\end{equation}
\blue{where $d_{m+1}=1$.}
It's obvious to check that the nonlinearities $g_i, i=1,\cdots,m+1$ satisfy the assumption \eqref{A1}. Moreover, due to the definition $w_i(x, t)=e^{-k_1t}u_i(x, t)$, it follows from \eqref{A3} and \eqref{eq-def-g} the growth control
\begin{equation}\label{eq-g_i}
  \begin{split}
    |g_i(x,t,\widetilde{w})|&\leq e^{-k_1t}(|f_i(x,t,u)|+k_1e^{k_1t}|w_i|)\\
    &\leq e^{-k_1t}(C(1+|u|^{3+\varepsilon})+k_1e^{k_1t}|w_i|)\\
    &\leq Ce^{(2+\varepsilon)k_1T_{\max}}(1+|\widetilde{w}|^{3+\varepsilon}).
  \end{split}
\end{equation}
Moreover, the nonlinearities of \eqref{eq-w} satisfies the condition \eqref{A2'}, i.e.
\begin{equation}\label{new7}
    \sum^{m+1}_{i=1}g_i(x,t,\widetilde{w})=k_0e^{-k_1t}
\end{equation}
thanks to \eqref{eq-m+1}. Now we can apply the results of Theorem \ref{th1.2} to get that \eqref{eq-w} has a global classical solution $\widetilde{w}$. Changing back to the original unknowns $u_i(x, t) =e^{k_1t}w_i(x, t)$ for $i=1,\cdots, N$, we obtain finally the global existence of classical solution to \eqref{eq1.1}.

\medskip
If, additionally, $k_0 = k_1 = 0$, the condition \eqref{new7} becomes
\begin{equation*}
	\sum_{i=1}^{m+1}g_i(x,t,\widetilde{w}) = 0.
\end{equation*}
Therefore, Theorem \ref{th1.2} implies that
\begin{equation*}
	\sup_{i=1,\ldots, m}\sup_{t\geq 0}\|w_i(t)\|_{\LO{\infty}} < +\infty,
\end{equation*}
which in turn gives, since with $k_1 = 0$, $w_i(x,t) = u_i(x,t)$, the uniform-in-time bound of solutions to \eqref{eq1.1}. The proof of Theorem \ref{th1.1} is, therefore, finished.
\end{proof}

\section{Cubic intermediate sum condition}\label{sec3}
\subsection{Constant diffusion coefficients - Proof of Theorem \ref{th1.3}}\label{sub3.1}

We first show that the entropy dissipation condition \eqref{E} implies the boundedness of solution in \blue{$L\log L(\Omega):=\{u: \Omega\to\mathbb{R}, ~\int_{\Omega}|u|\log|u|\,dx <+\infty\}$} and consequently in $L^1(\Omega)$.
\begin{lemma}\label{lem-1}
	Assume \eqref{E}. Then, there exists constant $C>0$ depending only on $\|u_0\|_{\LO{\infty}}$ and $|\Omega|$, such that for any \blue{$T\in(0, T_{\max})$,}
	\begin{equation*}
		\|(u_i\log u_i)(t)\|_{\LO{1}} \leq C\bra{1+e^{CT}} \quad \text{ for all } \quad t\in [0,T).
	\end{equation*}
	As a consequence,
	\begin{equation*}
		\|u_i(t)\|_{\LO{1}} \leq C\bra{1+e^{CT}} \quad \text{ for all } \quad t\in [0,T).
	\end{equation*}	
\end{lemma}

\begin{proof}	
From the entropy inequality \eqref{E} we have
\begin{equation*}
\begin{split}
\frac{d}{dt}\int_{\Omega}\sum^m_{i=1}&(u_i(\log u_i+\mu_i)-u_i)dx=\int_{\Omega}\sum^m_{i=1}(\log u_i+\mu_i)\partial_tu_idx\\
&=-\int_{\Omega}\sum^m_{i=1}d_i\frac{|\nabla u_i|^2}{u_i}dx+\int_{\Omega}\sum^m_{i=1}(\log u_i+\mu_i)f_i(x,t,u)dx\\
&\leq -4\int_{\Omega}\sum^m_{i=1}d_i|\nabla \sqrt{u_i}|^2+ k_2\sumi\int_{\Omega}\bra{ u_i\log u_i + (\mu_i-1)u_i}dx + k_3|\Omega|\\
&\leq  k_2\sumi\int_{\Omega}\bra{ u_i\log u_i + (\mu_i-1)u_i}dx + k_3|\Omega|.
\end{split}
\end{equation*}
By standard Gronwall's inequality we have
\begin{equation*}
\begin{split}
\sumi\int_{\Omega}&\bra{ u_i(t)\log u_i(t)+ (\mu_i-1)u_i(t)}dx\\
&\leq \sumi\int_{\Omega}\bra{ u_{i0}\log u_{i0} + (\mu_i-1)u_{i0}}dx + k_3|\Omega|e^{k_2T} \quad \forall t\in[0,T).
\end{split}
\end{equation*}
We rewrite this inequality as
\blue{\begin{equation}\label{eq3.1}
\begin{split}
\sumi\int_{\Omega}(u_i(t)\log u_i(t)&-u_i(t)+1) dx\leq \sumi\int_{\Omega}\bra{u_{i0}\log u_{i0}+ (\mu_i-1)u_{i0}}dx \\
&+ k_3|\Omega|e^{k_2T}+m|\Omega|-\sumi\int_{\Omega}\mu_i u_i(t)dx \quad \forall t\in[0,T).
\end{split}
\end{equation}}
Using the inequality $x\log x-x+1\geq Lx-e^L+1$ for all $L>0$ we have
\blue{\begin{equation*}
\begin{split}
\sumi\int_{\Omega}( u_i(t)\log u_i(t)-u_i(t)+1) \,dx\geq k\sumi\int_{\Omega}u_i(t)dx-e^km|\Omega|+m|\Omega|\,dx
\end{split}
\end{equation*}}
with $k=2\max_{i=1,\cdots,m}|\mu_i|$. Therefore we obtain the estimate
\begin{equation*}
\begin{split}
\sumi\int_{\Omega}u_i(t)dx&\leq \frac{2}{3k}\bra{\sumi\int_{\Omega}\bra{ u_{i0}\log u_{i0} + (\mu_i-1)u_{i0}}dx + k_3|\Omega|e^{k_2T}+e^km|\Omega|} \\
&\leq C\bra{1+e^{CT}}\quad \text{ for all } \quad t\in [0,T).
\end{split}
\end{equation*}
Which, together with the positivity of the solution, leads to the uniform in time $L^1-$bound. The bound of $\|(u_i\log u_i)(t)\|_{\LO{1}}$ immediately \blue{follows from \eqref{eq3.1}} and $x\log x\geq x-1$ for all $x\geq 0$, i.e.
\begin{equation*}
\begin{split}
\|(u_i\log u_i)(t)\|_{\LO{1}}\leq C\bra{1+e^{CT}}\quad \text{ for all } \quad t\in [0,T).
\end{split}
\end{equation*}
\end{proof}

The following modified interpolation inequality is important for the sequel.
\begin{lemma}[A modified Gagliardo-Nirenberg inequality]\label{MGN}
For any $\varepsilon>0$ and $\Omega\subset \mathbb{R}^n$, there exists $c_\varepsilon>0$ such that for all $f\in H^1(\Omega)$,
\begin{equation*}
\begin{split}
\|f\|^{2+\frac{2}{n}}_{\LO{2+\frac{2}{n}}}\leq \varepsilon \|f\|^2_{H^1(\Omega)}\|f\log|f|\|^{\frac{2}{n}}_{\LO1}+c_\varepsilon \|f\|_{L^1(\Omega)}.
\end{split}
\end{equation*}
\end{lemma}
\begin{remark}
	This inequality was proved in \cite{tang2018global} for the case $n=1$ and $n=2$.
\end{remark}

\begin{proof}
Fix a constant $N>1$. Define a function $\chi:\mathbb{R}\rightarrow \mathbb{R}$ as $\chi(s)=0$ if $|s|\leq N,~\chi(s)=2(|s|-N)$ when $N<|s|<2N$ and $\chi(s)=|s|$ when $|s|>2N.$ In this proof we use
$$\Omega\{|f|\geq N\}:=\{x\in\Omega:|f(x)|\geq N\},$$
$$\Omega\{|f|\leq N\}:=\{x\in\Omega:|f(x)|\leq N\}$$
and
$$\Omega\{|f|\leq 2N\}:=\{x\in\Omega:|f(x)|\leq 2N\}.$$

First we write
\begin{equation}\label{M1}
\begin{split}
\|f\|^{2+\frac{2}{n}}_{\LO{2+\frac{2}{n}}}&\leq \bra{\|\chi(f)\|_{L^{2+\frac{2}{n}}(\Omega)}+\||f|-\chi(f)\|_{\LO{2+\frac{2}{n}}}}^{2+\frac{2}{n}}\\
&\leq 2^{1+\frac{2}{n}}\bra{\|\chi(f)\|^{2+\frac{2}{n}}_{\LO{2+\frac{2}{n}}}+\||f|-\chi(f)\|^{2+\frac{2}{n}}_{\LO{2+\frac{2}{n}}}}
\end{split}
\end{equation}
and then estimate each term separately.  It is easy to see that
\begin{equation}\label{M2}
\begin{split}
\||f|-\chi(f)\|^{2+\frac{2}{n}}_{\LO{2+\frac{2}{n}}}&=\int_{\Omega}||f|-\chi(f)|^{2+\frac{2}{n}}dx\\
&\leq (2N)^{1+\frac{2}{n}}\int_{\Omega\{|f|\leq2N\}}|f|dx\leq(2N)^{1+\frac{2}{n}}\|f\|_{L^{1}(\Omega)}.
\end{split}
\end{equation}
Concerning the other term, we use the usual Gagliardo-Nirenberg inequality
\begin{equation}\label{M3}
\begin{split}
\|\chi(f)\|^{2+\frac{2}{n}}_{\LO{2+\frac{2}{n}}}\leq C\|\chi(f)\|^2_{H^{1}(\Omega)}\|\chi(f)\|^{\frac{2}{n}}_{\LO1},
\end{split}
\end{equation}
for some $C>0$ depending only on $n$ and $\Omega$. On the one hand
\begin{equation}\label{M4}
\begin{split}
\|\chi(f)\|^2_{H^{1}(\Omega)}=\|\chi'(f)\nabla f\|^2_{\LO2}+\|\chi(f)\|^2_{\LO2}\leq 4\|f\|^2_{H^1(\Omega)}
\end{split}
\end{equation}
and on the other hand
\begin{equation}\label{M5}
\begin{split}
\|\chi(f)\|^{\frac{2}{n}}_{\LO1}&\leq\bra{\int_{\Omega\{|f|\geq N\}}|f|dx}^{\frac{2}{n}}\\
&\leq\frac{1}{(\log N)^{\frac{2}{n}}}\|f\log |f|\|_{\LO1}^{\frac{2}{n}}.
\end{split}
\end{equation}

By combining \eqref{M1}-\eqref{M5} we obtain
\begin{equation*}
\begin{split}
\|f\|^{2+\frac{2}{n}}_{\LO{2+\frac{2}{n}}}\leq 2^{3+\frac{2}{n}}\frac{C}{(\log N)^{\frac{2}{n}}} \|f\|^2_{H^1(\Omega)}+2^{1+\frac{2}{n}}(2N)^{1+\frac{2}{n}}\|f\|_{\LO1}.
\end{split}
\end{equation*}

At this point we can choose $N$ to be large enough to obtain the desired inequality.
\end{proof}

\begin{remark}
	\blue{A similar argument can be used if $f\log|f|$ in Lemma \ref{MGN} is replaced by $f\Phi(f)$ for some strictly increasing function $\Phi: (0,\infty)\to \R$ satisfying $\lim_{z\to 0}(z\Phi(z)) = 0$ and $\lim_{z\to\infty }\Phi(z) = +\infty$. We leave the details for the interested reader.}
\end{remark}

We need another result from \cite{morgan2021global}, which provides an important consequence of the intermediate sum condition \eqref{A4}, which in turn becomes essential in building appropriate $L^p$-energy functions.
\begin{lemma}\cite[Lemma 2.2]{morgan2021global}\label{lem3.3}
	\blue{Assume \eqref{A1} and \eqref{A4}}. Then there exist componentwise increasing functions $g_i: \R^{m-i} \to \R_+$ for $i=1,\ldots, m-1$ such that: if $\theta = (\theta_1,\ldots, \theta_m)\in (0,\infty)^m$ satisfies $\theta_i \geq g_i(\theta_{i+1},\ldots, \theta_m)$ for all $i=1,\ldots, m-1$, then
	\begin{equation*}
		\sumi \theta_i f_i(x,t,u) \leq K_\theta\bra{1+\sumi u_i}^r
	\end{equation*}
	for some constant $K_\theta$ depending on $\theta$, $g_i$.
\end{lemma}

The following  \blue{$L^p$-energy function} has been developed in \cite{morgan2021global,fitzgibbon2021reaction}. For any $2\leq p \in \mathbb N$, we define
\begin{equation}\label{Lpenergy}
	E_p[u]:= \sum_{\beta\in \mathbb{Z}_+^m, |\beta| = p}\begin{pmatrix}p\\ \beta\end{pmatrix}\theta^{\beta^2}u^{\beta}
\end{equation}
and
\begin{equation}\label{EE}
	\mathscr{E}_p[u]:= \intO E_p[u(x)]dx,
\end{equation}
where
\begin{equation*}
	\begin{pmatrix}p\\ \beta\end{pmatrix} = \frac{p!}{\beta_1!\ldots \beta_m!}, \quad \theta^{\beta^2} = \prod_{j=1}^m \theta_j^{\beta_j^2}, \quad u^{\beta} = \prod_{j=1}^{m}u_j^{\beta_j}.
\end{equation*}

\begin{lemma}\label{lem1}
	\blue{Assume \eqref{A1} and \eqref{A4}.} Then for any $2\le p \in \mathbb N$, there exists $\theta = (\theta_1,\ldots, \theta_m)\in (0,\infty)^m$ such that the energy function $\EE_p[u]$ defined in \eqref{EE} satisfies
	\begin{equation*}
		\frac{d}{dt}\EE_p[u] + \alpha_p\sumi \intO |\pa_{x}(u_i^{p/2})|^2dx \leq C\bra{1+\sumi \intO u_i^{p-1+r}dx}
	\end{equation*}
	along the trajectory of solution to \eqref{eq1.1}, for some constant $\alpha_p>0$ depending on $\theta$ and $p$, \blue{$C=C(p,\theta)$}.
\end{lemma}

\begin{proof}
Let $u$ solve \eqref{eq1.1} and $\EE_p(t):=\EE_p[u](t)$ be defined in \eqref{EE}. Then by \cite[Lemma A.1]{morgan2021global}, we have
  \begin{equation*}
     \begin{split}
         \EE'_p(t)&=\intO\sum_{|\beta|=p-1}\begin{pmatrix}p\\ \beta\end{pmatrix}\theta^{\beta^2}u(x,t)^{\beta}\sum^m_{j=1}\theta_j^{2\beta_j+1}\frac{\partial}{\partial t}u_j(x,t)dx\\
         &= \intO\sum_{|\beta|=p-1}\begin{pmatrix}p\\ \beta\end{pmatrix}\theta^{\beta^2}u(x,t)^{\beta}\sum^m_{j=1}\theta_j^{2\beta_j+1}\bra{d_j\pa_{xx}u_j+f_j(x,t,u)}dx.
      \end{split}
  \end{equation*}
If we  apply \cite[Lemma A.2]{morgan2021global}  and integration by parts, we have
  \begin{equation}\label{I}
     \begin{split}
         \intO\sum_{|\beta|=p-1}&\begin{pmatrix}p\\ \beta\end{pmatrix}\theta^{\beta^2}u(x,t)^{\beta}\sum^m_{j=1}\theta_j^{2\beta_j+1}d_j\pa_{xx}u_jdx\\
         &=-\intO\sum_{|\beta|=p-2}\begin{pmatrix}p\\ \beta\end{pmatrix}\theta^{\beta^2}u(x,t)^{\beta}\sumi\sum^m_{j=1}a_{ij}\pa_x u_i\pa_x u_jdx\\
         &=:(I)
      \end{split}
  \end{equation}
with
  \begin{equation*}
     a_{ij}=\begin{cases}
         \frac{d_i+d_j}{2}\theta_i^{2\beta_i+1}\theta_j^{2\beta_j+1} &\quad \text{if}\quad i\neq j,\\
         d_i\theta_i^{4\beta_i+4} &\quad \text{if}\quad i= j.
      \end{cases}
  \end{equation*}
  Denote by $\mathscr A = (a_{i,j})_{i,j=1,\ldots, m}$, $\mathscr{C} = \text{diag}(\theta_i^{-2\beta_i - 1})$, and
  \begin{equation*}
  		\mathscr{M} = (m_{ij}):= \begin{cases}
	  		d_i\theta_i^2, & \text{ when } i=j,\\
	  		\frac{d_i+d_j}{2}, &\text{ when } i\ne j.
  		\end{cases}
  \end{equation*}
  It's easy to check that
  \begin{equation*}
  	\mathscr A = \mathscr{C}^{-1}\mathscr{M}\mathscr{C}^{-1}.
  \end{equation*}
  By choosing $\theta_i, i=1,\ldots, m$ large enough, the matrix $\mathscr{M}$ is diagonally dominant, and therefore positive definite. This implies that $\mathscr A$ is also positive definite. Therefore, there exists $\omega>0$ such that
  \begin{equation*}
  	(\pa_x u)^{\top}\mathscr{A}(\pa_x u) \geq \omega|\pa_x u|^2
  \end{equation*}
  where $\pa_x u = (\pa_x u_i)_{i=1,\ldots,m}$ is \blue{a column vector} in $\R^m$. It then follows from \eqref{I} that
  \begin{equation*}
     \begin{aligned}
        (I)&=-\intO\sum_{|\beta|=p-2}\begin{pmatrix}p\\ \beta\end{pmatrix}\theta^{\beta^2}u(x,t)^{\beta}(\pa_x u)^{\top} \mathscr{A}(\pa_x u)dx\\
        &\leq \alpha_p\sum_{i=1}^m\intO \left|\pa_x(u_i^{\frac p2})\right|^2dx.
      \end{aligned}
  \end{equation*}
Consequently, returning to above, we obtain
  \begin{equation}\label{eq-E}
     \begin{split}
         \EE'_p(t)&+\alpha_p\sumi\intO|\pa_x(u_i^{\frac{p}{2}})|^2dx\\
         &\leq \intO\sum_{|\beta|=p-1}\begin{pmatrix}p\\ \beta\end{pmatrix}\theta^{\beta^2}u(x,t)^{\beta}\sumi\theta_i^{2\beta_i+1}f_i(x,t,u)dx.
      \end{split}
  \end{equation}

So, we choose the components of $\theta = (\theta_i)$ inductively so that $\theta_i$ are sufficiently large that the previous positive definiteness condition of $\mathscr{M}$ is satisfied, and
  \begin{equation}\label{star}
     \begin{split}
         \theta_i\geq g_i(\theta_{i+1}^{2p-1},\cdots,\theta_{m}^{2p-1}) \quad\text{for}\quad i=1,\cdots,m-1,
      \end{split}
  \end{equation}
where $g_i$ are functions constructed in Lemma \ref{lem3.3}. Note that $\theta_i \leq \theta_i^{2\beta_i + 1} \leq \theta_i^{2p - 1}$. Since $g_i$ is componentwise increasing,  the relation \eqref{star} implies
\begin{equation*}
	\theta_i^{2\beta_i + 1} \geq g_i\bra{\theta_{i+1}^{2\beta_{i+1}+1}, \ldots, \theta_m^{2\beta_m + 1}}, \quad \forall i=1,\ldots, m-1.
\end{equation*}
Now we can apply Lemma \ref{lem3.3}, to obtain some $K_{\widetilde{\theta}}$  so that for all $\beta\in\mathbb{Z}_+$ with $|\beta|=p-1$, we have
  \begin{equation*}
     \begin{split}
         \sumi \theta_{i}^{2\beta_i+1}f_i(x,t,u)\leq K_{\widetilde{\theta}}\bra{1+\sumi u_i^r} \quad\text{for all}\quad (x,t,u)\in \Omega\times\mathbb{R}_+\times \mathbb{R}^m_+.
      \end{split}
  \end{equation*}
\blue{It follows that there exists} $C_p>0$, such that \eqref{eq-E} implies
  \begin{equation*}
     \begin{split}
         \EE'_p(t)&+\alpha_p\sumi\intO|\pa_x(u_i^{\frac{p}{2}})|^2dx\leq C_p\bra{\sumi\intO u_i^{p-1+r}dx+1}.
      \end{split}
  \end{equation*}
\end{proof}

\begin{lemma}\label{lem3.5}
\blue{Assume \eqref{E},  \eqref{A1} and \eqref{A4}} with $r=3$. Then for any $T\in (0,T_{\max})$,
	\begin{equation*}
		\sup_{i=1,\ldots, m}\|u_i(t)\|_{\LO{2}} \leq C(T) \quad \text{ for all } \quad t\in (0,T),
	\end{equation*}
	with $C(T)$ depends continuously on $T\in (0,\infty)$.
\end{lemma}
\begin{proof}
	From Lemma \ref{lem1}, it follows by choosing $p=2$ that
	\begin{equation*}
		\frac{d}{dt}\EE_2[u]  + \alpha_2\sumi \intO|\pa_x u_i|^2dx \leq C\bra{1+\sumi \|u_i\|_{\LO{4}}^{4}}.
	\end{equation*}
	Therefore,
	\begin{equation*}
		\frac{d}{dt}\EE_2[u]  + \alpha_2\sumi \|u_i\|_{H^1(\Omega)}^2 \leq C\bra{1+\sumi \|u_i\|_{\LO{4}}^{4}}.
	\end{equation*}	
	Fix $\eps>0$. Thanks to Lemmas \ref{lem-1} and \ref{MGN},
	\begin{equation*}
		\sumi \|u_i\|_{\LO{4}}^4 \leq \eps\sumi \|u_i\|_{H^1(\Omega)}^2 + C_\eps.
	\end{equation*}
	By choosing $\eps>0$ small enough, it follows that
	\begin{equation*}
		\frac{d}{dt}\EE_2[u] + \frac{\alpha_2}{2}\sumi \|u_i\|_{H^1(\Omega)}^2 \leq C,
	\end{equation*}
	which finishes the proof of Lemma \ref{lem3.5}.
\end{proof}

\begin{lemma}\label{lem-Lp}
	\blue{Assume \eqref{E},  \eqref{A1} and \eqref{A4}} with $r=3$. Then for any $2\leq p \in \mathbb N$ and any $T\in (0,T_{\max}]$,
	\begin{equation}\label{eq-Lp}
		\sup_{i=1,\ldots, m}\|u_i(t)\|_{\LO{p}} \leq C(T,p) \quad \text{ for all } \quad t\in (0,T)
	\end{equation}
	with $C(T , p)$ depends continuously on $T>0$.
\end{lemma}

\begin{proof}
  By using the $E_p[u]$ defined in \eqref{Lpenergy} and Lemma \ref{lem1}, we obtain
	\begin{equation}\label{new8}
		\frac{d}{dt}\EE_p[u] + \alpha_p\sumi \intO |\pa_{x}(u_i^{p/2})|^2dx \leq C\bra{1+\sumi \intO u_i^{p+2}dx}.
	\end{equation}
	Since we are in one dimension and have $\|u_i(t)\|_{\LO{2}} \leq C(T)$, we can apply \cite[Lemma 2.3]{morgan2021global} to obtain
		\begin{equation}\label{new9}
			\intO u_i^{p+2}dx \leq \eps\|u_i^{p/2}\|_{H^1(\Omega)}^2 + C_\eps(T).
		\end{equation}

	By adding $\alpha_p\sumi \intO u_i^pdx$ to both sides of \eqref{new8} and using \eqref{new9} with $\eps = \alpha_p/2$, it yields
	\begin{equation*}
		\frac{d}{dt}\EE_p[u] + \frac{\alpha_p}{2}\sumi \intO u_i^{p}dx \leq C(T).
	\end{equation*}
	This leads to
	\begin{equation*}
		\frac{d}{dt}\EE_p[u] + \alpha\EE_p[u] \leq C(T)
	\end{equation*}
for some constant $\alpha>0$. Gronwall's lemma yields
	\blue{\begin{equation*}
		\sup_{t\in (0,T)}\EE_p[u(t)]\leq C_1(T),
	\end{equation*}}
which implies \eqref{eq-Lp}.
\end{proof}

We now can prove the global existence part of Theorem \ref{th1.3}
\begin{proof}[\textbf{Proof of Theorem \ref{th1.3} - Global existence}]
The existence of local in time, classical, nonnegative solutions $u=(u_1,u_2,\cdots,u_m)$ on some maximal time interval $[0,T_{\max})$ follows from classical results.

Our aim in the following is to prove \blue{for all finite $T\in(0,T_{\max}]$,}
$$\limsup_{t\rightarrow T_{}}\|u_i(t)\|_{L^\infty(\Omega)}<\infty, \quad \text{for all}~i=1,2,\cdots,m. \quad \Rightarrow T_{\max}=+\infty.$$	
From the polynomial bound \eqref{polynomial}, \blue{we have by the comparison principle that} $u_i \le v_i$ where $v_i$ solves
 \begin{equation*}
	 \begin{cases}
		 \pa_t v_i - d_i\pa_{xx}v_i = h_i:= C\bra{1+\sum_{j=1}^m u_j^{\ell}}, &(x,t)\in Q_T,\\
		 \pa_xv_i(0,t) = \pa_xv_i(L,t) = 0, &t\in (0,T),\\
		 v_i(x,t) = u_{i,0}(x), &x\in\Omega.
	 \end{cases}
 \end{equation*}
 Now thanks to Lemma \ref{lem-Lp}, $h_i \in L^p(Q_T)$ for any $1\leq p <\infty$. \blue{By the smoothing effect of} the heat operator $\pa_t - d_i\pa_{xx}$, it follows that for some $p>\frac{3}{2}$,
 \begin{equation*}
 	\|v_i\|_{L^{\infty}(Q_T)} \le C\bra{\|u_{i,0}\|_{\LO{\infty}}, \|h_i\|_{\LQ{p}}}.
 \end{equation*}
 Thus
 $$\|u_i\|_{L^\infty(Q_T)}\leq C_T,$$
 and consequently, $T_{\max}  = +\infty$.
\end{proof}

\medskip
We now turn to the uniform-in-time bounds for the case \blue{when $k_2 = k_3 = 0$ or $k_2<0$ in \eqref{E}, i.e. the case of entropy dissipation. In order to do that, it's sufficient to prove there exists $C>0$ independent of $\tau>0$ such that}
\begin{equation}\label{1}
\begin{split}
\sup_{\tau\in \mathbb{N}}\|u_i\|_{L^\infty(\Omega\times(\tau,\tau+1))}\leq C.
\end{split}
\end{equation}

Since $k_2=k_3=0$, we have from \eqref{E}
$$\sumi f_i(x,t,u)(\log u_i+\mu_i)\leq0.$$
This implies the following uniform bounds.
\begin{lemma}\label{lem-U1}
	Assume \eqref{E} with $k_2 = k_3 = 0$ \blue{or $k_2<0$}. Then, there exists constant $C>0$ depending only on $\|u_0\|_{\LO{\infty}}$ and $|\Omega|$, such that
	\begin{equation*}
		\sup_{t>0}\|u_i(t)\|_{\LO{1}} \leq C \quad \text{and} \quad  \sup_{t>0}\|(u_i\log u_i)(t)\|_{\LO{1}} \leq C.
	\end{equation*}	
\end{lemma}

\begin{proof}	
From the entropy inequality \eqref{E} we have
\begin{equation*}
\begin{split}
\frac{d}{dt}\int_{\Omega}\sum^m_{i=1}&(u_i(\log u_i+\mu_i)-u_i)dx=\int_{\Omega}\sum^m_{i=1}(\log u_i+\mu_i)\partial_tu_idx\\
&=-\int_{\Omega}\sum^m_{i=1}d_i\frac{|\nabla u_i|^2}{u_i}dx+\int_{\Omega}\sum^m_{i=1}(\log u_i+\mu_i)f_i(x,t,u)dx\le 0.
\end{split}
\end{equation*}
Integration on $(0,t)$ provides
\begin{equation*}
\begin{split}
\sumi\int_{\Omega}&\bra{ u_i(t)\log u_i(t)+ (\mu_i-1)u_i(t)}dx\leq \sumi\int_{\Omega}\bra{ u_{i0}\log u_{i0} + (\mu_i-1)u_{i0}}dx .
\end{split}
\end{equation*}
We rewrite this inequality as
\begin{equation}\label{eq-U3.1}
\begin{split}
\sumi\int_{\Omega}( u_i(t)\log u_i(t)&-u_i(t)+1)dx\leq \sumi\int_{\Omega}\bra{ u_{i0}\log u_{i0}+ (\mu_i-1)u_{i0}}dx \\
&+m|\Omega|-\sumi\int_{\Omega}\mu_iu_i(t)dx, \quad \forall t\in[0,T).
\end{split}
\end{equation}
Using the inequality $x\log x-x+1\geq Kx-e^K+1$ for all $K>0$ we have
\blue{\begin{equation*}
\begin{split}
\sumi\int_{\Omega} (u_i(t)\log u_i(t)-u_i(t)+1)\,dx\geq k\sumi\int_{\Omega}u_i(t)dx-e^km|\Omega|+m|\Omega|
\end{split}
\end{equation*}}
with $k=2\max_{i=1,\cdots,m}|\mu_i|$. Therefore we obtain the estimate
\begin{equation*}
\begin{split}
\sumi\int_{\Omega}u_i(t)dx&\leq \frac{2}{3k}\bra{\sumi\int_{\Omega}\bra{ u_{i0}\log u_{i0} + (\mu_i-1)u_{i0}}dx +e^km|\Omega|} \\
&\leq C.
\end{split}
\end{equation*}
Which, together with the positivity of the solution, leads to the uniform in time $L^1$-bound. The bound of $\|(u_i\log u_i)(t)\|_{\LO{1}}$ follows immediately from \eqref{eq-U3.1} and $x\log x\geq x-1$ for all $x\geq 0$, i.e.
\begin{equation*}
\begin{split}
\|(u_i\log u_i)(t)\|_{\LO{1}}\leq C.
\end{split}
\end{equation*}
\end{proof}

\begin{lemma}\label{lem-U2}
	Assume \eqref{E} with $k_2 = k_3 = 0$ \blue{or $k_2<0$, \eqref{A1} and \eqref{A4}} with $r=3$. Then
	\begin{equation*}
		\sup_{t>0}\|u_i(t)\|_{\LO{2}} \leq C.
	\end{equation*}
\end{lemma}
\begin{proof}
	From Lemma \ref{lem1}, it follows by choosing $p=2$ that
	\begin{equation*}
		\frac{d}{dt}\EE_2[u]  + \alpha_2\sumi \intO|\pa_x u_i|^2dx \leq C\bra{1+\sumi \|u_i\|_{\LO{4}}^{4}}.
	\end{equation*}
	Therefore,
	\begin{equation*}
		\frac{d}{dt}\EE_2[u]  + \alpha_2\sumi \|u_i\|_{H^1(\Omega)}^2 \leq C\bra{1+\sumi \|u_i\|_{\LO{4}}^{4}}.
	\end{equation*}	
	Fix $\eps>0$. Thanks to Lemmas \ref{lem-U1} and \ref{MGN},
	\begin{equation*}
		\sumi \|u_i\|_{\LO{4}}^4 \leq \eps\sumi \|u_i\|_{H^1(\Omega)}^2 + C_\eps.
	\end{equation*}
	By choosing $\eps>0$ small enough, it follows that
	\begin{equation*}
		\frac{d}{dt}\EE_2[u] + \frac{\alpha_2}{2}\sumi \|u_i\|_{H^1(\Omega)}^2 \leq C,
	\end{equation*}
	where the constant $C$ is independent of time $t>0$. By using $\sumi \|u_i\|_{H^1(\Omega)}^2 \geq \beta_0\EE_2[u]$ for some $\beta_0>0$, we obtain
	\begin{equation*}
		\frac{d}{dt}\EE_2[u] + \frac{\alpha_2\beta_0}{2}\EE_2[u] \leq C.
	\end{equation*}
	With the classical Gronwall lemma, we can finish the proof of Lemma \ref{lem-U2}.
\end{proof}

\begin{lemma}\label{lem-U5}
	Assume \eqref{E} with $k_2 = k_3 = 0$ \blue{or $k_2<0$, \eqref{A1} and \eqref{A4}} with $r=3$. Then for any $2\leq p \in \mathbb N$
	\begin{equation}\label{eq-U3.3}
		\sup_{t>0}\|u_i(t)\|_{\LO{p}} \leq C(p).
	\end{equation}
\end{lemma}

\begin{proof}
  By using the $E_p[u]$ defined in \eqref{Lpenergy} and Lemma \ref{lem1}, we obtain
	\begin{equation*}
		\frac{d}{dt}\EE_p[u] + \alpha_p\sumi \intO |\pa_{x}(u_i^{p/2})|^2dx \leq C\bra{1+\sumi \intO u_i^{p+2}dx}.
	\end{equation*}
	Adding $\alpha_p\sumi \intO u_i^p$ to both sides and using Young's inequality yield
	\begin{equation}\label{j1}
		\frac{d}{dt}\EE_p[u] + \alpha_p\sumi \|u_i^{p/2}\|_{H^1(\Omega)}^2 \leq C\bra{1+\sumi \intO u_i^{p+2}dx}.
	\end{equation}
	Thanks to Lemma \ref{lem-U2}, we can apply \cite[Lemma 2.3]{morgan2021global} to estimate for any $\delta>0$
	\begin{equation*}
		\intO u_i^{p+2} \leq \delta\|u_i^{p/2}\|_{H^1(\Omega)}^2 + C_\delta \quad \forall i=1,\ldots, m.
	\end{equation*}
	By choosing $\delta = \alpha_p/2$, inserting this into \eqref{j1}, and using $\sumi \|u_i^{p/2}\|_{H^1(\Omega)}^2 \geq \beta_1\EE_p[u]$ for some $\beta_1>0$, we get
	\begin{equation*}
		\frac{d}{dt}\EE_p[u] + \frac{\alpha_p \beta_1}{2}\EE_p[u] \leq C,
	\end{equation*}
	where the constant $C$ is \textit{independent of the time $t>0$}. The classical Gronwall lemma yields the desired bound \eqref{eq-U3.3}.
\end{proof}
\begin{proof}[\textbf{Proof of Theorem \ref{th1.3}-- Uniform-in-time bounds}]
In order to prove \eqref{1}, we define an cut-off function $\varphi_\tau$ as
\begin{equation*}
\varphi_\tau(t)=
\begin{cases}
0  \quad\text{for on}  \quad[0,\tau],\\
1  \quad\text{for on}  \quad[\tau+1,\infty)
\end{cases}
\end{equation*}
and $0\leq\varphi_\tau'\leq2$. Direct computations leads to
\begin{equation*}
\begin{cases}
\partial_t(\varphi_\tau u_i)-d_i\partial_{xx}(\varphi_\tau u_i)=\varphi'_\tau u_i+\varphi_\tau f_i(x,t,u),&(x,\tau)\in \Omega\times(\tau,\tau+2),\\
\pa_x(\vpt u_i)(0,t) = \pa_x(\vpt u_i)(L,t) = 0, &t\in (\tau,\tau+2),\\
(\vpt u_i)(x,\tau)=0, &x\in \Omega.
\end{cases}
\end{equation*}
Thanks to Lemma \ref{lem-U5} and polynomial upper bound \eqref{polynomial}, we have
\begin{equation*}
\varphi'_\tau u_i+\varphi_\tau f_i(x,t,u)\leq G_i(x,t,u):=C\bra{1+\sum^m_{k=1}u^\ell_k}.
\end{equation*}
Thanks to Lemma \ref{lem-U5} for any $2\leq p<\infty$ a constant $C_p>0$ such that
\begin{equation*}
\|G_i(x,t,u)\|_{L^p(\Omega\times (\tau,\tau+2))}\leq C_p,\quad \forall i=1,\cdots,m.
\end{equation*}
where $C_p$ is a constant independent of $\tau \in \N$.
Therefore, by the smoothing effect of the heat operator $\pa_t - d_i\pa_{xx}$ and comparison principle we get
\begin{equation*}
\|\varphi_\tau u_i\|_{L^\infty(\Omega\times (\tau,\tau+2))}\leq C,\quad \forall i=1,\cdots,m,
\end{equation*}
where $C$ is a constant independent of $\tau\in \mathbb{N}$. Thanks to $\varphi_\tau\geq 0$ and $\varphi|_{(\tau+1,\tau+2)}\equiv1$, we obtain finally the uniform-in-time bound
\begin{equation*}
\sup_{\tau\in \mathbb{N}}\|u_i\|_{L^\infty(\Omega\times(\tau,\tau+1))}\leq C,\quad \forall i=1,\cdots,m,
\end{equation*}
and the proof of Theorem \ref{th1.3} is complete.
\end{proof}

\subsection{Discontinuous diffusion coefficients - Proof of Theorem \ref{th1.4}}\label{sub3.2}

Due to the discontinuity of the diffusion coefficients, we do not expect to obtain strong solutions to \eqref{discontinuous_diffusion}. Therefore, we will work with weak solutions whose concept is defined in the following
\begin{definition}
  A vector of non-negative state variables $u=(u_1,\cdots,u_m)$ is called a weak solution to \eqref{discontinuous_diffusion} on $(0,T)$ if
  \begin{equation*}
     u_i\in C([0,T];L^2(\Omega))\cap L^2(0,T;H^1_0(\Omega)),~f_i(x,t,u)\in L^2(0,T;L^2(\Omega)),
  \end{equation*}
  with $u_i(\cdot,0)=u_{i,0}(\cdot)$ for all $i=1,\cdots,m$, and for any test function $\varphi\in L^2(0,T;H^1_0(\Omega))$ with $\partial_t\varphi\in L^2(0,T;H^{-1}(\Omega))$, one has
  \begin{equation*}
  \begin{split}
    \intO u_i(x,T)\varphi(x,T)dx-\int^T_0\intO u_i\pa_t\varphi dxdt&+\int^T_0\intO D_i(x,t)\nabla u_i\cdot \nabla\varphi dxdt\\
    &=\intO u_{i,0}\varphi(x,0)dx + \int^T_0\intO f_i(x,t,u)\varphi dxdt.
  \end{split}
  \end{equation*}
\end{definition}
\begin{proof}[\textbf{Proof of Theorem \ref{th1.4}}]
We start with a truncated system where the nonlinearities are regularized to be bounded. Then crucial estimates which are uniform with respect to the regularization parameters are derived using the $L^p$-energy estimates as in the case of constant diffusion coefficients. Of importance are the bounds in $\LQ{\infty}$ of the approximate solutions. The global existence of solutions follows straightforwardly thanks to these bounds and passing to the limit of the regularization. Finally, the uniform-in-time bounds are obtained similarly to the proof of Theorem \ref{th1.3}.

\medskip

For any $\varepsilon>0$, we consider the truncated system: for all $i=1,\cdots,m$,
\begin{equation}\label{eq-truncated}
  \begin{cases}
    \pa_t u_i^{\varepsilon} - \nabla_x\cdot (D_i(x,t)\nabla_x u_i^{\varepsilon}) = f_i^{\varepsilon}(x,t,u^{\varepsilon}), & x\in\Omega, \; t>0,\\
	D_i(x,t)\nabla_x u_i^{\varepsilon}(x,t)\cdot \nu(x) = 0, & x\in\pa\Omega, t>0,\\
	u_{i}^{\varepsilon}(x,0) = u_{i,0}^{\varepsilon}(x), & x\in\Omega,
  \end{cases}
\end{equation}
where
  \begin{equation*}
  \begin{split}
    f_i^{\varepsilon}(x,t,u^{\varepsilon}):=\frac{f_i(x,t,u^{\varepsilon})}{[1+\varepsilon\sum^m_{j=1}|f_j(x,t,u^{\varepsilon})|]^{-1}}.
  \end{split}
  \end{equation*}
and $u_{i,0}^{\varepsilon}\in \LO{\infty}$ such that $\|u_{i,0}^{\varepsilon}-u_{i,0}\|_{\LO{\infty}}\stackrel{\varepsilon\rightarrow0}{\longrightarrow}0$. Note that since $u_{i,0}\in\LO{\infty}$ one can take $u_{i,0}^{\varepsilon}=u_{i,0}$ to obtain the global existence. Following \cite[Lemma 2.1]{fitzgibbon2021reaction}, for any fixed $\eps>0$, there exists a global bounded, non-negative weak solution to \eqref{eq-truncated}.

\medskip

We would like to emphasize that all constants in this subsection are independent of $\varepsilon$. By using the observation
\begin{align*}
	\sumi f_i^{\eps}(u^\eps)(\log u_i^\eps + \mu_i) &= \frac{1}{1+\eps\sumi |f_i(u^\eps)|}\sumi f_i(u^\eps)(\log u_i^\eps + \mu_i)\\
	&\leq \frac{1}{1+\eps\sumi |f_i(u^\eps)|}\bra{k_2\sumi u_i(\log u_i + \mu_i - 1) + k_3}\\
	&\leq C\sumi u_i(\log u_i + \mu_i - 1) + C.
\end{align*}
We obtain, similarly to Lemma \ref{lem-1}
	\begin{equation*}
		\|(u^{\varepsilon}_i\log u^{\varepsilon}_i)(t)\|_{\LO{1}} \leq C\bra{1+e^{CT}} \quad \text{ for all } \quad t\in [0,T),
	\end{equation*}
	and as a consequence,
	\begin{equation*}
		\|u^{\varepsilon}_i(t)\|_{\LO{1}} \leq C\bra{1+e^{CT}} \quad \text{ for all } \quad t\in [0,T).
	\end{equation*}	

We establish next the estimates in $L^p$-norm by using the $L^p$-energy estimates similarly to the case of constant diffusion coefficients. Thanks to \cite[Lemma 2.3]{fitzgibbon2021reaction}, we get similar results to Lemma \ref{lem3.3} where $f_i$ are replaced by $f_i^{\eps}$, i.e.
\begin{equation*}
	\sumi \theta_i f_i^{\eps}(x,t,u^\eps) \leq K_\theta\bra{1+\sumi u_i^\eps}^r.
\end{equation*}
This allows us to repeat the arguments in Lemmas \ref{lem1} and \ref{lem3.5} to obtain
\begin{equation}\label{L2-ueps}
	\sup_{i=1,\cdots,m}\|u_i^\eps(t)\|_{\LO{2}} \leq C(T), \quad \forall t\in (0,T).
\end{equation}
We remark that the discontinuous diffusion coefficients complicate the integration by parts when differentiating $\EE_p(t)$, but this issue is overcome by using the results in \cite[Lemma 4.2]{fitzgibbon2021reaction}. From \eqref{L2-ueps} and the modified Gagliardo-Nirenberg inequality in Lemma \ref{MGN}, we can repeat the arguments in Lemma \ref{lem-Lp} to get for any $2\leq p \in \mathbb N$,
\begin{equation*}
	\sup_{i=1,\ldots, m}\|u_i^{\eps}(t)\|_{\LO{p}} \leq C(T,p), \quad \forall t\in (0,T).
\end{equation*}
This is enough to obtain bounds in $\LQ{\infty}$ of the approximate solutions $u^\eps$ to system \eqref{eq-truncated} which are \textit{uniform in $\eps>0$}. It follows that
  \begin{equation*}
		\|f_i(x,t,u^{\varepsilon})\|_{\LQ {\infty}}\leq C_T, \quad \forall i=1,\cdots,m.
  \end{equation*}
 By multiplying \eqref{eq-truncated} by $u^{\varepsilon}_i$ then integrating on $Q_T$ gives
  \begin{equation*}
  \begin{split}
    \frac{1}{2}\|u^{\varepsilon}_i(T)\|^2_{\LO2}&+\int^T_0\intO D_i(x,t)\nabla u^{\varepsilon}_i\cdot \nabla u^{\varepsilon}_idxdt\\
    &=\frac{1}{2}\|u^{\varepsilon}_{i,0}\|^2_{\LO2}+\int^T_0\intO f^{\varepsilon}_i(u^{\varepsilon})u^{\varepsilon}_idxdt.
  \end{split}
  \end{equation*}
  Using \eqref{eq-di} we obtain
  \begin{equation*}
  \begin{split}
    \|u^{\varepsilon}_i(T)\|^2_{\LO2}+2\lambda\|\nabla u^{\varepsilon}_i\|_{\LQ2}\leq C_T, \quad \forall i=1,\cdots,m.
  \end{split}
  \end{equation*}
  Thus
  \begin{equation*}
  \begin{split}
   \{u^{\varepsilon}_i\}_{\varepsilon\geq 0}~&\text{is bounded uniformly in $\varepsilon$ in }~\LQ{\infty}\cap L^2(0,T;H^1_0(\Omega)),\\
   \{\partial_tu^{\varepsilon}_i\}_{\varepsilon\geq 0}~ &\text{is bounded uniformly in $\varepsilon$ in}~L^2(0,T;H^{-1}(\Omega)).
  \end{split}
  \end{equation*}
  The classical Aubin-Lions lemma gives the strong convergence
  \begin{equation*}
  \begin{split}
   u_{i}^{\varepsilon}\stackrel{\varepsilon\rightarrow0}{\longrightarrow} u_{i}~\text{strongly in}~L^{2}(0,T;\LO2).
  \end{split}
  \end{equation*}
  Consequently, for any $1\leq p<\infty$,
  \begin{equation*}
  \begin{split}
   u_{i}^{\varepsilon}\stackrel{\varepsilon\rightarrow0}{\longrightarrow} u_{i}~\text{strongly in}~\LQ p,
  \end{split}
  \end{equation*}
  thanks to the uniform $L^{\infty}$-bound of $u_{i}^{\varepsilon}$. This is enough to pass to the limit in the weak formulation of \eqref{eq-truncated}
  \begin{equation*}
  \begin{split}
  -\intO \varphi(\cdot,0)u_{i,0}^{\varepsilon}dx &-\int^T_0\intO u_{i}^{\varepsilon}\partial_t\varphi dxdt+ \int^T_0\intO d_i(x,t)\nabla u_{i}^{\varepsilon}\cdot\nabla\varphi dxdt\\
    &= \int^T_0\intO f_i^{\varepsilon}(x,t,u^{\varepsilon})\varphi dxdt,
  \end{split}
  \end{equation*}
  to obtain that $u = (u_1,\cdots,u_m)$ is a global weak solution to \eqref{discontinuous_diffusion} and additionally
  \begin{equation*}
		\|u_i\|_{\LQ {\infty}}\leq C_T, \quad \forall i=1,\cdots,m.
  \end{equation*}

\medskip
The proof of uniform-in-time bounds follows similarly to that of Theorem \ref{th1.3}, so we omit the details here.
\end{proof}

	\medskip
\noindent{\bf Data Availability Statements.}
Data sharing not applicable to this article as no datasets were generated or analysed during the current study.

	\medskip
	\noindent{\bf Conflict of interest.} This paper has no conflict of interest.

	\medskip
	\noindent{\bf Acknowledgement.} B.Q. Tang is partially supported by NAWI Graz. C. Sun  are partially supported by NSFC Grants No. 12271227.  J. Yang are partially supported by NSFC Grants No. 12271227 and China Scholarship Council (Contract No. 202206180025).

	\medskip
	The authors sincerely thank the reviewers for their very careful reading of our manuscript and for providing valuable comments, which helped us to improve the presentation and readability of the paper.

\end{document}